\documentclass[hidelinks,onefignum,onetabnum]{siamart220329}

\usepackage{amsfonts}
\usepackage{graphicx}
\usepackage{overpic}
\usepackage{algorithm}
\usepackage{algpseudocode}
\usepackage{subcaption}
\usepackage{mathtools}
\usepackage{dsfont}
\usepackage{enumitem}

\algnewcommand{\algorithmicor}{\textbf{ or }}
\algnewcommand{\OR}{\algorithmicor}

\newcommand{\R}{\mathbb{R}}
\newcommand{\F}{\textup{F}}
\newcommand{\C}{\mathbb{C}}
\newcommand{\mx}{\mathbf{x}}
\newcommand{\my}{\mathbf{y}}
\renewcommand{\d}{\,\textup{d}}
\DeclareMathOperator{\sign}{sign}
\newcommand{\norm}[1]{\left\lVert#1\right\rVert}

\newsiamremark{remark}{Remark}
\newsiamremark{hypothesis}{Hypothesis}
\crefname{hypothesis}{Hypothesis}{Hypotheses}
\newsiamthm{claim}{Claim}

\headers{Multivariate rational approximation}{N. Boull\'e, A. Herremans, and D. Huybrechs}

\title{Multivariate rational approximation of functions with curves of singularities\thanks{Submitted to the editors \today.
\funding{N.B. was supported by an INI-Simons Postdoctoral Research Fellowship and the SciAI Center, funded by the Office of Naval Research (ONR) under Grant N00014-23-1-2729. A.H. is a PhD fellow of the Research Foundation Flanders (FWO), funded by grant 11P2T24N. D.H. was supported in part by FWO research project G088622N.}}}

\author{Nicolas Boull\'e\thanks{Department of Applied Mathematics and Theoretical Physics, University of Cambridge, Cambridge, CB3 0WA, UK
  (\email{nb690@cam.ac.uk}).}
\and Astrid Herremans\thanks{Department of Computer Science, KU Leuven, 3001 Leuven, Belgium
  (\email{astrid.herremans@kuleuven.be}, \email{daan.huybrechs@kuleuven.be}).}
\and Daan Huybrechs\footnotemark[3]}

\usepackage{amsopn}

\begin{document}

\maketitle

\begin{abstract}
  Functions with singularities are notoriously difficult to approximate with conventional approximation schemes. In computational applications, they are often resolved with low-order piecewise polynomials, multilevel schemes, or other types of grading strategies. Rational functions are an exception to this rule: for univariate functions with point singularities, such as branch points, rational approximations exist with root-exponential convergence in the rational degree. This is typically enabled by the clustering of poles near the singularity. Both the theory and computational practice of rational functions for function approximation have focused on the univariate case, with extensions to two dimensions via identification with the complex plane. Multivariate rational functions, i.e., quotients of polynomials of several variables, are relatively unexplored in comparison. Yet, apart from a steep increase in theoretical complexity, they also offer a wealth of opportunities. A first observation is that singularities of multivariate rational functions may be continuous curves of poles, rather than isolated ones. By generalizing the clustering of poles from points to curves, we explore constructions of multivariate rational approximations to functions with curves of singularities.
\end{abstract}



\section{Introduction}

Rational approximations of functions provide a powerful representation of functions with singularities in comparison with polynomial approximation. Theoretical results pointing in this direction date back to Zolotarev~\cite{zolotarev1877application} and Newman~\cite{newman1964}. More recently, several numerical algorithms have been developed which take advantage of these long known approximation results; these include the adaptive Antoulas--Anderson (AAA) algorithm~\cite{driscoll2024aaa,huybrechs2023aaa,nakatsukasa2018aaa,nakatsukasa2020algorithm,xue2023computation}, and the ``lightning'' method, which approximates solutions to differential equations on two-dimensional domains with corners using univariate rational functions~\cite{baddoo2020lightning,brubeck2022lightning,gopal2019new,gopal2019solving}.

In one dimension one can employ a partial fraction representation with a polynomial part to represent a function $f$ defined on an interval $I\subset\mathbb{R}$ as~\cite{herremans2023resolution}:
\begin{equation} \label{eq_1d_rat}
  f(z)\approx r(z) = \sum_{j=1}^{N_q}\frac{a_j}{z-q_j}+\sum_{j=0}^{N_p}b_j P_j(z).
\end{equation}
This representation lies at the heart of the lightning method. Contrary to the AAA algorithm, which aims to solve the non-linear problem of best approximation by a rational function, the finite poles $\{q_j\}_{j=1}^{N_q}$ in \cref{eq_1d_rat} are fixed and selected a priori to cluster exponentially towards the singularities of the function $f$, of which the locations are assumed to be known. This results in a linear approximation problem, which achieves root-exponential convergence to functions with isolated singularities~\cite{gawlik2019zolotarev,gawlik2020zolotarev,gawlik2021approximating,huybrechs2023sigmoid}, where the rate of convergence depends on the pole locations and the type of singularity~\cite{gopal2019solving,trefethen2021clustering,herremans2023resolution}. Here, singularities are the points in the domain where the function either is discontinuous, e.g.\ $f(z)=\sign(z)$, has a discontinuity in one of its derivatives, e.g.\ $f(z)=|z|$, or is non-differentiable, e.g.\ $f(z)=\sqrt{z}$. The addition of a smooth residual part consisting of polynomials $P_j(z)$ up to degree $N_p$ in the rational approximation, as in~\cref{eq_1d_rat}, has been shown to improve the overall accuracy and numerical stability of the scheme~\cite{herremans2023resolution}. Hence, \cite{herremans2023resolution} observed the convergence of \cref{eq_1d_rat} to functions $x^\alpha$ on $[0,1]$ at optimal rates~\cite{stahl1993best} by using least-square fitting. However, these approaches are usually limited to univariate rational expressions and the approximation of functions with isolated singularities.

In this work, we are interested in computing multivariate rational approximations to functions with known curves of singularities. Inspired by the lightning approximation, we employ two-dimensional analogs of \cref{eq_1d_rat} with clustering curves of poles. There are several applications motivating this study. First, the representation of curves of singularities is a major topic in image processing, leading to research on optimal sparse representations of piecewise smooth functions delineated by $C^2$-curves achieved, e.g., by curvelets~\cite{candes2002curvelets}. Second, the field of scientific machine learning aims to solve partial differential equations (PDEs) or learn solution operators associated with PDEs from data using a neural network representation~\cite{karniadakis2021physics}. These applications often involve the approximation of functions with shocks or discontinuities~\cite{mao2020physics}. Recently, a neural network architecture based on rational activation functions, called rational neural networks~\cite{boulle2020rational}, has been proposed to approximate Green's functions associated with linear PDEs~\cite{boulle2022data}. When approximating Green's functions, rational neural networks, which are multivariate rational approximants, have poles that tend to cluster towards the diagonal of the domain, where Green's functions are singular~\cite{evans2010partial}. The development of multivariate rational approximation algorithms to functions with curves of singularity should deepen our understanding of this phenomenon.

To the best of our knowledge, there have been limited explorations of computational techniques for multivariate rational approximation beyond~\cite{austin2020practical,berrut2021linear,cuyt2010practical,cuyt1983multivariate,cuyt1985multivariate,yang2023fast}. Most of the previous studies in this topic only consider functions with isolated singularities. Theoretical convergence rates for multivariate rational approximations are derived in \cite{devore1986multivariate} by piecing together local polynomial approximations with rational partition-of-unity functions, extending the univariate ideas introduced by DeVore~\cite{devore11983maximal}. The series of works by Cuyt and coauthors~\cite{cuyt2010practical,cuyt1983multivariate,cuyt1985multivariate} construct multivariate Pad\'e approximations of the form $r(x,y) = P(x,y)/Q(x,y)$ and shows practical error bounds. More recently,~\cite{austin2020practical} computes multivariate Pad\'e approximations iteratively to discard potential spurious poles of the denominator $Q(x,y)$. This approach suffers when the function has singularities, and determining polynomial degrees of numerator and denominator is based on a heuristic method, which impacts the convergence rates.

\subsection{Organization of the paper}

The main result of this paper is a linear scheme to construct multivariate rational approximations to functions that are singular along the zero levels of a multivariate polynomial. The scheme is based on well-chosen polynomial denominators associated with level curves of that same polynomial with complex shifts, in combination with exponential clustering of those shifts towards zero. Significant attention goes towards a numerically stable implementation of that scheme. We describe a series of examples with increasing levels of complexity.

The paper is organized as follows. First, we introduce a multivariate rational approximation scheme for functions with singularities located along straight lines using a tensor-product representation in \cref{sec_tensor}. We prove the robustness of the scheme with respect to the nature of the singularity. We then generalize this construction to functions with singularities along the diagonal of a domain using a piecewise rational approximation method in \cref{sec_piecewise}. Finally, we consider the approximation of functions with algebraic curves of singularities in \cref{sec_curved} and conclude in \cref{sec:conclusions}.

\section{Tensor-product representation} \label{sec_tensor}

We consider a tensor-product representation of a multivariate rational function to approximate two-dimensional functions with singularities located along straight horizontal or vertical lines. This construction is a generalization of \cref{eq_1d_rat} with exponentially clustering lines of poles.

\subsection{Construction of the approximation} \label{sec_construction_tensor}

First,  we consider the rational approximation of a function $f\colon(0,1]\times [0,1]\to \mathbb{R}$ with a singularity located at $x=0$. We employ a basis of functions consisting of $N_q$ finite poles, with associated partial fractions $\{\frac{q_j}{x-q_j}\}_{j=1}^{N_q}$, and a polynomial basis of degree $N_p$: $\{P_0,\ldots,P_{N_p}\}$. Following~\cite{herremans2023resolution,trefethen2021clustering}, we choose tapered exponentially clustered poles,
\begin{equation} \label{eq_poles}
  q_j = \pm i \exp(-\sigma(\sqrt{N_q}-\sqrt{j})), \quad 1\leq j\leq N_q,
\end{equation}
yet we cluster them along the imaginary axes (note the factor $\pm i$). This allows us to cluster poles exponentially close to singularities of the function inside the domain without introducing additional poles. The parameter $\sigma>0$ in \cref{eq_poles} controls the spacing between poles and is chosen to be $\sigma=2\pi$ in this paper. This parameter choice is discussed later in~\cref{sec:parameterchoices}.

One can then approximate $f$ by the tensor-product of a rational function in $x$ with a polynomial in $y$ as follows:
\begin{equation} \label{eq_tensor_rational}
  f(x,y)\approx r(x,y) = \sum_{j=1}^{N_q}\frac{q_j}{x-q_j}\sum_{k=0}^{N_p}a_{j,k}P_k(y)+\sum_{0\leq j,k\leq N_p}b_{j,k} P_j(x)P_k(y).
\end{equation}
We adopt the ``lightning + polynomial approximation'' scheme of~\cite{herremans2023resolution}, which consists of choosing the polynomial degree as $N_p = \mathcal{O}(\sqrt{N_q})$. While the convergence rate of a multivariate polynomial approximation is dictated by the Euclidean degree of the multivariate polynomial rather than the maximal degree~\cite{trefethen2017multivariate}, this representation allows for an efficient least-square solver exploiting the tensor-product structure of~\cref{eq_tensor_rational}. To that end, the rational function $r$ can be written as
\[r(x,y) = \begin{bmatrix}
    \frac{q_1}{x-q_1} & \cdots & \frac{q_{N_q}}{x-q_{N_q}} & P_0(x) & \cdots & P_{N_p}(x)
  \end{bmatrix}
  C
  \begin{bmatrix}
    P_0(y) \\
    \vdots \\
    P_{N_p}(y)
  \end{bmatrix},\]
where the matrix $C\in \C^{(N_q+N_p+1)\times (N_p+1)}$ contains the coefficients $\{a_{j,k}\}$, $\{b_{j,k}\}$ as follows,
\[C =
  \begin{bmatrix}
    a_{1,0}   & \cdots & a_{1,N_p}   \\
    \vdots    & \ddots & \vdots      \\
    a_{N_q,0} & \cdots & a_{N_q,N_p} \\
    b_{0,0}   & \cdots & b_{0,N_p}   \\
    \vdots    & \ddots & \vdots      \\
    b_{N_p,0} & \cdots & b_{N_p,N_p}
  \end{bmatrix}.
\]
We then solve the following least-square system to find the coefficient matrix $C$:
\begin{equation} \label{eq_least_square}
  \min_{C\in\C^{(N_q+N_p+1)\times (N_p+1)}}\|ACB^\top-F\|_{\textup{F}},
\end{equation}
where $F\in \C^{M_x\times M_y}, \; A \in \C^{M_x\times (N_q + N_p + 1)}$, and $B \in \C^{M_y\times (N_p + 1)}$ are matrices containing function values of $f$ at a set of sampling points $\{(x_i,y_j)\}_{i,j}$, function values of the basis functions in the $x$ variable at $\{x_i\}_i$, and function values of the basis functions in the $y$ variable at $\{y_j\}_j$, respectively.

The $M_x\times M_y$ sampling points are chosen to be a superposition of a product grid of $M_p\times M_p$ Chebyshev points in the $x$ and $y$ coordinates, and a product of $M_q$ points clustering exponentially fast to $x=0$ with $M_p$ Chebyshev points in the $y$ coordinates (see~\cref{fig_sample_pts}(a)). Thus, $M_x = M_p+M_q$ and $M_y = 2M_p$. We typically choose a certain amount of oversampling,\footnote{Oversampling ensures that the size of the residual of the linear system reflects the continuous approximation error. The choices of oversampling factors in our experiments are such that this is always the case (in particular the factors are chosen large enough).}
\[
  M_q = 3 N_q \quad \mbox{and} \quad M_p = 2 N_p.
\]
The clustering points $\{x_i\}_{i=1}^{M_q}$ are logarithmically spaced in order to resolve the singularity of the function $f$ at $x=0$, our practical implementation being
\[x_i = 10^{-16 + 16 (i-1) / (M_q-1)}, \quad 1\leq i\leq M_q.\]
The total number of sampling points is $M_x\times M_y = 2 M_p^2 + 2M_pM_q$.

\begin{figure}[tbp]
  \centering
  \begin{overpic}[width=\textwidth]{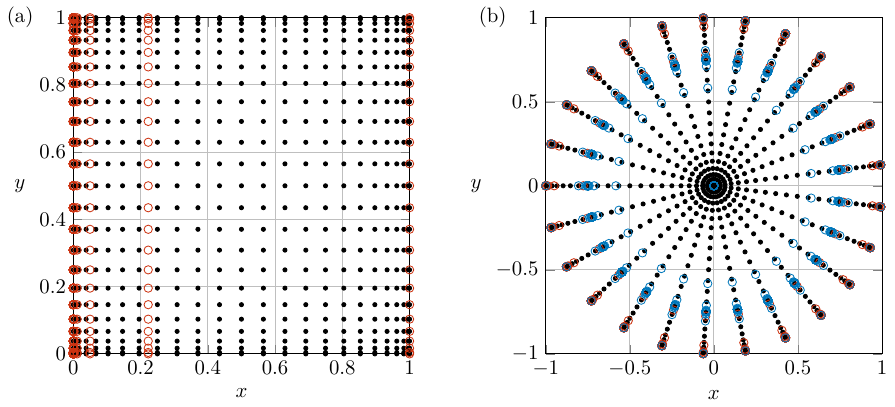}
  \end{overpic}
  \caption{(a) The product grid of sample points used for evaluating a function defined on a square domain. The sample grid is a superposition of product Chebyshev points in the $x$ and $y$ coordinates shown as black dots, and a product of points clustering exponentially fast to $x=0$ with Chebyshev points in the $y$ coordinates, highlighted in red. (b) Sample grid for a function defined on the unit disk with singularities located at $r=3/4$ and $r=1$. The grid is constructed using Chebyshev points in the radial direction and trigonometric points in the angular direction, along with points clustering towards $r=3/4$ (in blue) and $r=1$ (in red) in the radial direction.}
  \label{fig_sample_pts}
\end{figure}

The tensor-product method described in this section generalizes naturally when the function has several straight-line singularities located at different points in the domain, including a mix of horizontal and vertical lines. As an example, for singularities located at $x=x_0$ and $y=y_0$, the rational approximant has the form
\begin{equation} \label{eq_tensorrational}
  r(x,y) = \begin{bmatrix}
    R(x,x_0) & P(x)
  \end{bmatrix}
  C
  \begin{bmatrix}
    R(y,y_0)^\top \\
    P(y)^\top
  \end{bmatrix},
\end{equation}
where the term $R(z,z_0)$ contains the basis of rational functions whose poles cluster exponentially close to $z=z_0$ as $R(z,z_0) = \begin{bmatrix} \frac{p_1}{(z - z_0) - q_1} & \cdots & \frac{p_{N}}{(z - z_0)-q_{N_1}} \end{bmatrix}$, and $P(z) = \begin{bmatrix}P_0(z) &\cdots & P_{N_p}(z) \end{bmatrix}$. Here, $C$ is a square matrix of $(N_q+N_p+1)\times (N_q+N_p+1)$ coefficients. Similar constructions can be made involving more terms $R(x,x_k)$ and $R(y,y_k)$ when more singularities are present. The sample grid is constructed as a superposition of Chebyshev points in the $x$ and $y$ coordinates, along with points clustering exponentially fast to each singularity (see~\cref{fig_sample_pts}(a)).

A variant of this scheme when the function is periodic in one direction can employ trigonometric polynomials in the rational representation~\eqref{eq_tensor_rational}, along with trigonometric (i.e., equispaced) sampled points. For example,~\cref{fig_sample_pts}(b) displays the sample grid for a function defined on the unit disk with singularities located at $r=3/4$ and $r=1$. The grid is constructed using Chebyshev points in the radial direction and trigonometric points in the angular direction (as black dots), along with points clustering towards $r=3/4$ (in blue) and $r=1$ (in red) in the radial direction.

\subsection{Regularized least squares fitting}
In general, the system matrices $A$ and $B$ introduced in~\cref{eq_least_square} are heavily ill-conditioned, yet accurate approximations can still be found using regularization techniques~\cite{adcock2019frames,adcock2020fna2,herremans2023resolution}. We regularize and solve~\cref{eq_least_square} with the truncated singular value decomposition (SVD) algorithm~\cite{engl1996regularization,hansen2013least,neumaier1998solving}. Due to the tensor-product structure of the rational representation~\cref{eq_tensor_rational}, the problem~\cref{eq_least_square} is equivalent to the following linear least squares problem:
\begin{equation} \label{eq:linearizedLS}
  \min_{\mathbf{c} \in \mathbb{C}^{(N_q+N_p+1)(N_p+1)}} \norm{(B \otimes A)\mathbf{c} - \mathbf{f}}_2,
\end{equation}
where $\mathbf{c} = \text{vec}(C)$ and $\mathbf{f} = \text{vec}(F)$ are the vectors obtained by stacking the columns of $C$ and $F$. \cref{alg_TSVD} solves this least squares problem using the truncated SVD of the matrix $B \otimes A$. Note that the algorithm employs $B^\top$ rather than $B^*$ even when $B$ is complex-valued and that $\odot$ denotes the Hadamard product of matrices.

\begin{algorithm}
  \caption{Truncated SVD solver}
  \label{alg_TSVD}
  \begin{algorithmic}[1]
    \Require $A \in \mathbb{C}^{M_A \times N_A}$, $B \in \mathbb{C}^{M_B \times N_B}$, $F \in \mathbb{C}^{M_A \times M_B}$, relative tolerance $\epsilon > 0$
    \Ensure $C \in \mathbb{C}^{N_A \times N_B}$ such that $ACB^\top \approx F$
    \State Compute the SVD of $A$ and $B$, $A = U_A\Sigma_A V_A^*$ and $B = U_B\Sigma_B V_B^*$.
    \State Compute $S =  \text{diag}(\Sigma_A^\dagger) (\text{diag}(\Sigma_B^\dagger))^\top$.
    \State Set all entries of $S$ larger than $1/(\epsilon \kern1pt \sigma_{1}(A) \kern1pt \sigma_{1}(B))$ to zero.
    \State Return $C = V_A (S \odot (U_A^*F(U_B^*)^\top)) V_B^\top$.
  \end{algorithmic}
\end{algorithm}

We emphasize that \cref{alg_TSVD} is computationally efficient and exploits the tensor-product approximation of the rational representation~\cref{eq_tensor_rational} as it only requires to compute the singular value decomposition of the matrices $A$ and $B$ instead of the SVD of the large Kronecker product $B\otimes A$. When the rational or polynomial degree is large, one can alternatively reduce the computational cost by computing the truncated SVD of $A$ and $B$ using the randomized SVD~\cite{halko2011finding,martinsson2020randomized}.

The ill-conditioning of the matrices $A$ and $B$, and hence of $B \otimes A$, is a consequence of the (near-)redundancy of the proposed basis functions in~\cref{eq_tensor_rational}.
It indicates that the error on the coefficients can be arbitrarily large; however, this does not mean that a small residual cannot be obtained. We provide a bound on the residual in~\cref{th:tsvd} by exploiting an existing result~\cite[Lem.~3.3]{coppe2020az} on the accuracy of regularized least squares fitting. The typical result in ill-conditioned least squares fitting is that solutions with small residuals can be found numerically if such a solution exists with a modest coefficient norm. In addition, a high level of oversampling ensures that a small residual corresponds to high approximation accuracy.

\begin{theorem}\label{th:tsvd}
  Let $A \in \mathbb{C}^{M_A \times N_A}$, $B \in \mathbb{C}^{M_B \times N_B}$, $F \in \mathbb{C}^{M_A \times M_B}$, and choose a relative threshold parameter $0<\epsilon < 1$. Let $C\in \mathbb{C}^{N_A\times N_B}$ be the coefficient matrix computed by Algorithm~\ref{alg_TSVD_sep}. Then,
  \[
    \norm{F - ACB^\top}_{\textup{F}} \leq \inf_{X \in \mathbb{C}^{N_A \times N_B}} \left\{ \norm{F-AXB^\top}_{\textup{F}} + \epsilon \norm{A}_2 \norm{B}_2 \norm{X}_{\textup{F}} \right\}.
  \]
\end{theorem}

\begin{proof}
  We first prove that~\cref{alg_TSVD} solves~\cref{eq:linearizedLS}, and hence~\cref{eq_least_square}, using the truncated SVD of $B \otimes A$.
  \cite{van1993approximation} shows that the singular value decomposition of $B \otimes A$ follows from the SVD of $A$ and $B$ as
  \[
    B \otimes A = (U_B \otimes U_A) (\Sigma_B \otimes \Sigma_A) (V_B \otimes V_A)^*.
  \]
  Consider the diagonal matrix $\Sigma = \Sigma_B \otimes \Sigma_A$ containing the singular values of $B \otimes A$:
  \[
    \{ \sigma_k(A) \sigma_l(B) \; \vert \; 1 \leq k \leq \min(M_A,N_A),  \; 1 \leq l \leq \min(M_B,N_B) \}.
  \]
  All singular values smaller than $\epsilon \sigma_1(A) \sigma_1(B)$ are set to zero in~\cref{alg_TSVD}, resulting in a new matrix denoted by $\Sigma_\epsilon$.
  A coefficient vector $\mathbf{c}$ can now be computed using the Moore--Penrose pseudoinverse, denoted by $\dagger$, of the truncated SVD of $B \otimes A$,
  \[
    \mathbf{c} = (V_B \otimes V_A) \Sigma_\epsilon^\dagger (U_B \otimes U_A)^* \mathbf{f}.
  \]
  This can again be condensed into a matrix equation using the dense matrix $S \in \mathbb{C}^{\min(M_A, N_B) \times \min(M_B, N_B)}$ whose entries contain diag($\Sigma_\epsilon$), i.e., the (truncated) singular values of $B \otimes A$.
  Also using the identities $(A \otimes B)^* = A^* \otimes B^*$ and $\text{vec}(AXB^\top) = (B \otimes A)\text{vec}(X)$ we obtain
  \[
    C = V_A (S \odot (U_A^* F (U_B^*)^\top))V_B^\top.
  \]
  This proves that~\cref{alg_TSVD} computes a solution to the least squares problem~\cref{eq:linearizedLS} using the truncated SVD of $B \otimes A$.
  A bound on the residual follows by applying \cite[Lem.~3.3]{coppe2020az} with an absolute truncation threshold $\varepsilon = \epsilon \norm{A}_2 \norm{B}_2$,
  \[
    \| \mathbf{f} - (B \otimes A) \mathbf{c} \|_2 \leq \inf_{\mathbf{x} \in \mathbb{C}^{N_A N_B}}\{ \| \mathbf{f} - (B \otimes A) \mathbf{x} \|_2 + \epsilon \|{A}\|_2 \|{B}\|_2 \| \mathbf{x} \|_2 \},
  \]
  which is equivalent to
  \[
    \| F - ACB^\top \|_\F \leq \inf_{X \in \mathbb{C}^{N_A \times N_B}}\{ \| F - AXB^\top \|_\F + \epsilon \|{A}\|_2 \|{B}\|_2 \| X \|_\F \},
  \]
  where we used the fact that $\norm{X}_\F =  \norm{\text{vec}(X)}_2=\norm{\mathbf{x}}_2$.
\end{proof}

\cref{th:tsvd} shows that the residual is of order $\epsilon$ for functions $f$ that can be approximated to within $\epsilon$ by a tensor-product approximant with bounded coefficients, i.e., high accuracy is obtained if there exist coefficients $X \in \mathbb{C}^{N_A \times N_B}$ such that $\norm{F-AXB^\top}_{\textup{F}} \approx \epsilon$ with $\norm{X}_{\textup{F}}\approx 1$. Thus, the condition number of the system does not adversely affect the accuracy of the calculations. The numerical examples in this paper use $\epsilon = 10^{-14}$. This parameter choice is discussed later in~\cref{sec:parameterchoices}.

\begin{remark}
  An alternative approach to solving the least squares problem is to regularize $A$ and $B$ independently, rather than the product $B \otimes A$. This results in a more straightforward algorithm, explained in~\cref{sec_appendix}, with similar theoretical guarantees (see~\cref{th:tsvd_sep}). However, this approach is less numerically stable for very small truncation thresholds, leading to our preference for \cref{alg_TSVD}.
\end{remark}

\subsection{Convergence analysis} \label{sec_theory}

In \cite[\S2]{gopal2019solving}, root-exponential convergence is proven for the lightning approximation to a univariate complex function on a polygonal subset of $\mathbb{C}$ with branch point singularities at the corners. We prove a similar result for the proposed bivariate approximant for functions on a rectangular subset of $\mathbb{R}^2$ with a continuum of singularities along one edge. To that end, we formulate a few lemmas in~\cref{sec_appendix2} that summarize and slightly extend the results of \cite{gopal2019solving}. The results imply the robustness of rational approximation with clustering poles with respect to the order of the singularity. Based on those results, we can show the existence of bivariate rational approximations to functions with singularities along the edge of a rectangle, even if the order of the singularity varies along the edge.

\begin{theorem}[Multivariate rational convergence analysis] \label{thm_main}
  Let $\Pi$ be the square domain $(0, 1] \times [0, 1]$ and $f:\Pi\to\C$ be an analytic function in $\Pi$, with the following three properties near the left segment $\{0\}\times [0,1]$:
  \begin{enumerate}[wide, labelwidth=0pt, labelindent=0pt]
    \item The boundary trace defined as $h(z_2) = f(0,z_2)$ is an analytic function of $z_2 \in [0,1]$.
    \item For each $z_2 \in [0,1]$, $f$ satisfies the following uniform growth condition:
          \[
            f(z_1,z_2) - h(z_2) = \mathcal{O}(\lvert z_1 \rvert^\delta), \quad \text{as } z_1 \to 0,
          \]
          for some $\delta > 0$ that is independent of $z_2$.
    \item For each $z_2 \in [0,1]$, the univariate function $f(\cdot,z_2)$ can be analytically continued to a disk near $z_1 = 0$ with a slit along the negative real axis.
  \end{enumerate}
  Then, for $1\leq n < \infty$, there exist a multivariate rational function $r_{n}$, of the form
  \begin{equation}
    r_{n}(z_1,z_2) = \sum_{j = 0}^{n-1} \frac{s_j(z_2)}{z_1-\beta_j} + p(z_1,z_2),
    \label{eq:approx}
  \end{equation}
  such that
  \[
    \| f - r_{n} \|_{\infty,\Pi} = \mathcal{O}(e^{-C\sqrt{n}}),\quad \text{as } n \to \infty,
  \]
  for some $C > 0$. Here, $s_j(z_2)$ is a univariate polynomial of degree scaling as $\mathcal{O}(\sqrt{n})$, $p(z_1,z_2)$ is a bivariate polynomial of degree scaling as $\mathcal{O}(\sqrt{n})$ as $n \to \infty$, and the poles $\beta_j$ are finite and exponentially clustered towards $0$, with arbitrary clustering parameter $\sigma > 0$ as defined by~\cref{eq:scaledpoles}.
\end{theorem}

\begin{proof}
  Since the function $h$ is analytic on $[0,1]$, and hence on a neighborhood in the complex $z_2$-plane, it can be approximated with exponential convergence by polynomials. Its approximation can be incorporated into the polynomial term $p(z_1,z_2)$, and it remains to describe the approximation of the function $u:\Pi\to \C$ defined as
  \[
    u(z_1,z_2) = f(z_1,z_2)-f(0,z_2) = f(z_1,z_2)-h(z_2),\quad (z_1,z_2)\in \Pi.
  \]
  Let $z_1\in (0,1]$, then the Chebsyhev projection of degree $K$ to the function $u(z_1,\cdot)$, where the Chebyshev polynomials $T_k$ are scaled to $[0,1]$, is~\cite[Thm.~3.1]{trefethen2019atap}
  \[
    \hat{u}(z_1,z_2) = \sum_{k=0}^{K} ~' a_k(z_1) T_k(z_2), \quad a_k(z_1) = \frac{2}{\pi}\int_{0}^1 u(z_1,z_2) \frac{T_k(z_2)}{\sqrt{z_2 - z_2^2}} \d z_2,\quad z_2\in [0,1],
  \]
  where $'$ indicates that the first term in the series is halved. The functions $a_k:(0,1]\to\C$ satisfy the conditions of~\cref{lem:gopal2019solving_lemma2} and can therefore be approximated with root-exponential accuracy by a degree $(n+m,n)$ rational function $\hat{a}_k$ as
  \[
    \hat{a}_k(z_1) = \sum_{j=0}^{n-1} \frac{s^k_j}{z_1-\beta_j} + p_k(z_1), \quad 0 \leq k \leq K,\quad z_1\in (0,1].
  \]
  Here, the poles $\beta_j$ defined by~\cref{eq:scaledpoles} are fixed, but the degree $m$ polynomial $p_k$ and the coefficients $s^k_j \in \mathbb{C}$ are dependent on $k$. Since $D$, $b$ and $\Omega$ introduced in~\cref{lem:gopal2019solving_lemma2} can be chosen such that they are independent of $k$, and also $\|a_k\|_{\infty,\Omega}$ can be bounded independently of $k$ (since the Chebyshev polynomials themselves are uniformly bounded), the error of each term is bounded by
  \[
    \norm{ a_k - \hat{a}_k }_{\infty,[0,1]} \leq Ae^{-C\sqrt{n}},
  \]
  for sufficiently large $n$ and for some constants $A,C > 0$. The fact that $\norm{a_k}_{\infty,\Omega}$ is bounded independently of $k$ also implies that one can choose the same degree $m = E\sqrt{n}$ for each $k$ for a fixed $E > 0$. Furthermore, we can choose the same $K$ for each $z_1 \in (0,1]$, based on the smallest Bernstein ellipse in which $u(z_1,\cdot)$ is analytic as a function of $z_2$, such that
  \begin{equation} \label{eq_approx_u_u_hat}
    \norm{u - \hat{u}}_{\infty,\Pi} = \mathcal{O}(e^{-C'K}),
  \end{equation}
  for some $C' > 0$ depending on the radius of that ellipse.

  Substituting $\hat{a}_k(z_1)$ into $\hat{u}(z_1,z_2)$ results in a bivariate rational function
  \begin{equation} \label{eq_def_q}
    q(z_1,z_2) \coloneqq \sum_{k=0}^{K} ~' \hat{a}_k(z_1) T_k(z_2) = \sum_{j=0}^{n-1} \frac{s_j(z_2)}{z_1 - \beta_j} + \sum_{k=0}^K p_k(z_1)T_k(z_2),
  \end{equation}
  where polynomial $s_j \coloneqq \sum_{k=0}^K s_j^k T_k$ has degree $K$ for $0\leq j\leq n-1$. The approximation error between $u$ and $q$ can be bounded uniformly in $\Pi$ using the triangular inequality,
  \[
    \norm{u - q}_{\infty,\Pi} \leq \norm{u - \hat{u}}_{\infty,\Pi} + \norm{\hat{u} - q}_{\infty,\Pi}.
  \]
  The first term on the right-hand side is bounded by~\cref{eq_approx_u_u_hat}. Considering the second term and the definition of $q$ in \cref{eq_def_q}, for $(z_1,z_2)\in \Pi$ we have
  \[
    |\hat{u}(z_1,z_2) - q(z_1,z_2)| \leq \sum_{k=0}^K |a_k(z_1) - \hat{a}_k(z_1)| |T_k(z_2)|,
  \]
  which we bound by taking the supremum over $z_1\in [0,1]$ and $z_2\in [0,1]$ as
  \begin{equation} \label{eq_approx_u_hat_q}
    \|\hat{u} - q\|_{\infty,\Pi} \leq \sum_{k=0}^K \|a_k - \hat{a}_k\|_{\infty,[0,1]} \norm{T_k}_{\infty,[0,1]} = \mathcal{O}(K e^{-C \sqrt{n}}).
  \end{equation}
  We now combine~\cref{eq_approx_u_u_hat,eq_approx_u_hat_q} and select $K = (C/C') \sqrt{n}$ such that
  \[\norm{u - q}_{\infty,\Pi} = \mathcal{O}\left(\sqrt{n}e^{-C\sqrt{n}}\right)=\mathcal{O}\left(e^{-C\sqrt{n}/2}\right),\]
  which concludes the proof.
\end{proof}

We add a few remarks to this result. First, we discuss the differences between the poles~\eqref{eq_poles} we use in practice and the poles~\cref{eq:scaledpoles} which we analyze. Following~\cite{gopal2019solving} we have analyzed poles with uniform exponential clustering. It was observed and explained in~\cite{trefethen2021clustering} that convergence rates can improve by a factor of 2 by using tapered exponential clustering of poles instead. The poles of~\eqref{eq_poles} use a tapered distribution for this reason. This is not a major difference, as by the same method of proof as ours and that of~\cite{gopal2019solving} essentially any set of points with $\sqrt{n}$ exponential scaling leads to root-exponential convergence and we make no claims about rates. Another minor difference is the scaling by the imaginary unit. Here, too, it is clear in the proof that scaling does not affect root-exponential convergence itself, only the precise convergence rate. Our choice of complex scaling allows us to approximate singularities within the domain of approximation. In the univariate case, a similar setting is analyzed in~\cite[\S5.1]{herremans2023resolution} in which clustered poles make an angle with respect to the approximation interval in the complex plane. More information on the choice of poles, including the spacing parameter $\sigma$, is given in~\cref{sec:parameterchoices}.


Second, although the basis of the approximation space has the structure of a tensor product, namely a rational function in $z_1$ times a polynomial in $z_2$, the space itself captures more interesting functions than low-rank functions. In particular, the space contains approximations to functions with continuously varying orders of algebraic singularities, as long as there is a lower bound $\delta > 0$ on the order.

Third, the result of this theorem is restricted to functions with singularities along one edge, but there is a straightforward generalization of the approximation space to deal with singularities along multiple edges. Indeed, the univariate rational approximations considered in~\cite{herremans2023resolution} have the form of partial fractions in combination with a polynomial term, hence a tensor product in the $x$ and $y$ variables leads to four terms in total. This is exactly the form we propose in~\cref{eq_tensorrational}. The results of~\cite{herremans2023resolution} immediately translate into favorable approximation results for separable functions such as $f(x,y) = \sqrt{x}\sqrt{y}$. However, due to the robustness of the clustering poles with respect to the nature of the singularity, the space can accurately capture a much broader class of functions. Particularly relevant classes of functions are solution sets of elliptic PDEs with analytic data, which give rise to edge and corner singularities, see~\cite[Sec.~1]{stephan1988singularities} and references therein. These function spaces have been extensively studied to investigate the convergence behavior of finite element methods with $hp$-refinement and neural networks, which are known to be related to rational approximation methods with exponentially clustered poles~\cite{huybrechs2023sigmoid}. Recent research~\cite{marcati2023exponential} proves exponential expressivity with stable ReLU neural networks for such classes on polytopal domains by constructing tensorized $hp$-approximations to point and edge singularities.

Finally, we note that the statement of the proof is about the accuracy of the approximation space, and not about the stability of the representation in partial fractions form. The latter is important in practice, and we will carefully select the normalization of the basis functions and rational and polynomial degrees such that the norms of the coefficients in the chosen basis are modest in size (see~\cite{adcock2019frames,adcock2020fna2} for the general theory and~\cite[\S4]{herremans2023resolution} for a discussion in the context of rational approximations). We do not actually prove that such representations always exist. A result along those lines is proven for the specific case of rational approximations to $\sqrt{z}$ in~\cite{herremans2023resolution}. We emphasize that the numerical stability of the representation and of the least squares approximation algorithm is the reason why approximations can be found to have very high accuracy for the examples in this paper, in some cases even close to machine precision.

\subsection{Numerical examples} \label{sec_numerical}

As a first numerical example, we consider the approximation of the function $f_1(x,y) = (x(1-x))^{\frac14 + y} \sqrt{y(1-y)}$ on the unit square. This function has singularities on all four edges of the square, with varying order. We approximate $f_1$ using a tensor-product rational function with poles clustering towards the boundary at $x=0$, $x=1$, $y=0$, and $y=1$, using the method introduced in \cref{sec_construction_tensor}. Here, we use a rational degree of $N_q=150$ along with a polynomial degree of $N_p=16$ for the smooth part of the approximation. We display in \cref{fig_tensor}(a) the function $f_1$ that we aim to approximate along with the approximation error $|f_1(x,y)-r(x,y)|$ by the rational function $r$ in \cref{fig_tensor}(b). We evaluate the rational approximant on an equispaced grid with $1000\times 1000$ points and obtain a maximum approximation error of $4.6\times 10^{-15}$. Moreover, the residual of the least-square problem \eqref{eq_least_square} has an infinity norm equal to $1.5\times 10^{-14}$, which is of the same order of magnitude as the approximation error, indicates that the error is uniformly small even exponentially close to the boundary.

\begin{figure}[tbp]
  \centering
  \begin{overpic}[width=\textwidth]{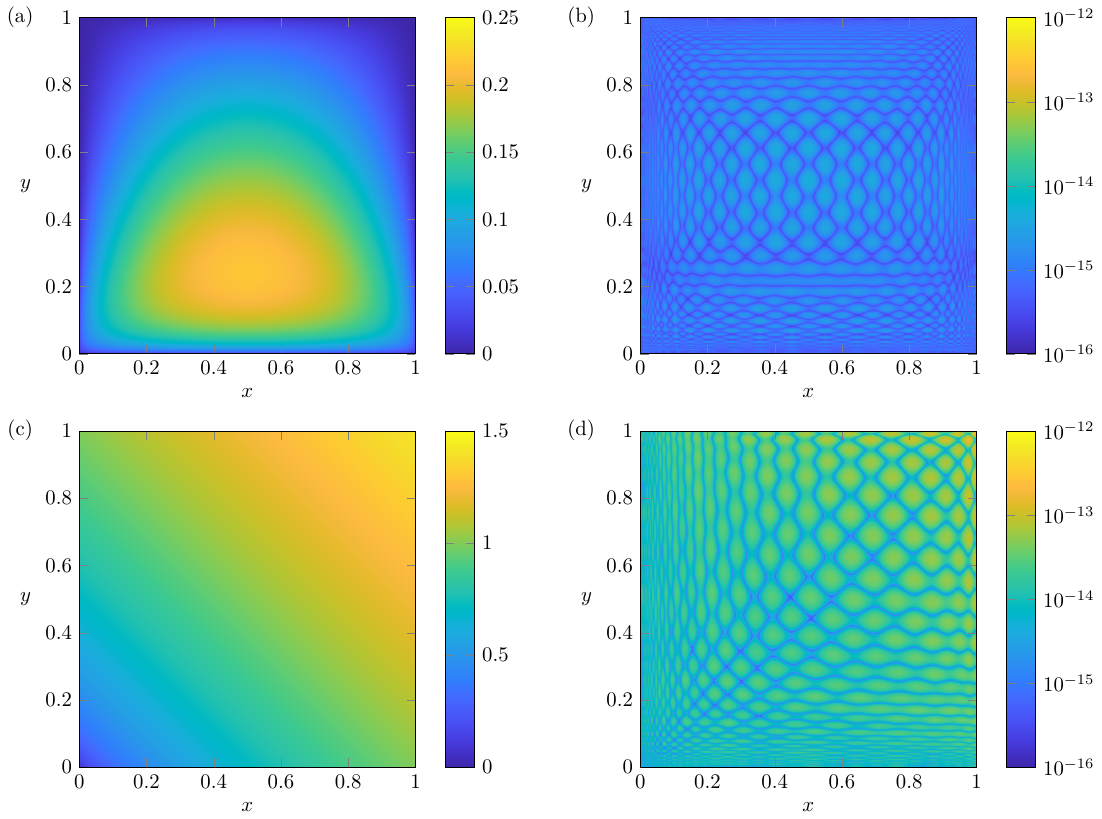}
  \end{overpic}
  \caption{(a) The function $f_1(x,y) = (x(1-x))^{\frac14 + y} \sqrt{y(1-y)}$ along with (b) the approximation error $|f_1(x,y)-r(x,y)|$ by a tensor-product rational function $r$ with lines of poles clustering towards the boundary of the square. (c)-(d) Same as (a)-(b) with the function $f_2(x,y)=\sqrt{x+y}$ and its rational approximant with poles located near the lines $x=0$ and $y=0$.}
  \label{fig_tensor}
\end{figure}

Then, we examine the function $f_2(x,y) = \sqrt{x+y}$ defined on $[0,1]^2$ with square-root singularity at $(0,0)$. We note that this function, plotted in \cref{fig_tensor}(c), is not separable as it cannot be written as a product of two one-dimensional functions. We visualize the approximation error with a rational function of degree $N_q= 150$ with lines of poles at $x=0$ and $y=0$, computed using the tensor-product method of \cref{sec_construction_tensor} in \cref{fig_tensor}(b). We observe that the approximation error is relatively uniform on the domain, with a maximum error of $1.6\times 10^{-13}$ on a $1000\times 1000$ equispaced grid. In \cref{fig_convergence}(a), we plot the maximum approximation error of the rational function with respect to the square root of the rational degree $N_q$ on a grid consisting of points that are distributed similarly as the sampling points, therefore containing points clustered towards the singularity, yet which is much finer than the sampling grid. The maximum error on this grid gives a reliable indication of the uniform error on $[0,1] \times [0,1]$. Here, we fix $N_p=16$ and $M_1=M_2=1200$ to observe the decay rate of the error. We find that the approximation error decays at a square-root exponential rate, i.e.,
\[\|f_2-r_{N_q}\|_{\max} = \mathcal{O}(e^{-1.16\sqrt{N_q}}),\]
as illustrated by the dashed line in \cref{fig_convergence}(a), which is in agreement with the theoretical results of \cref{sec_theory}. Next in \cref{fig_convergence}(b), we display the uniform approximation error between $f$ and its rational approximant as a function of the rational and polynomial degrees $N_q$ and $N_p$. We observe that the approximation error decays root-exponentially fast with respect to $N_q$ for a sufficiently high polynomial degree $N_p$. The figure confirms that a polynomial degree scaling as $N_p = \mathcal{O}(\sqrt{N_q})$ is sufficient.

\begin{figure}[tbp]
  \centering
  \begin{overpic}[width=\textwidth]{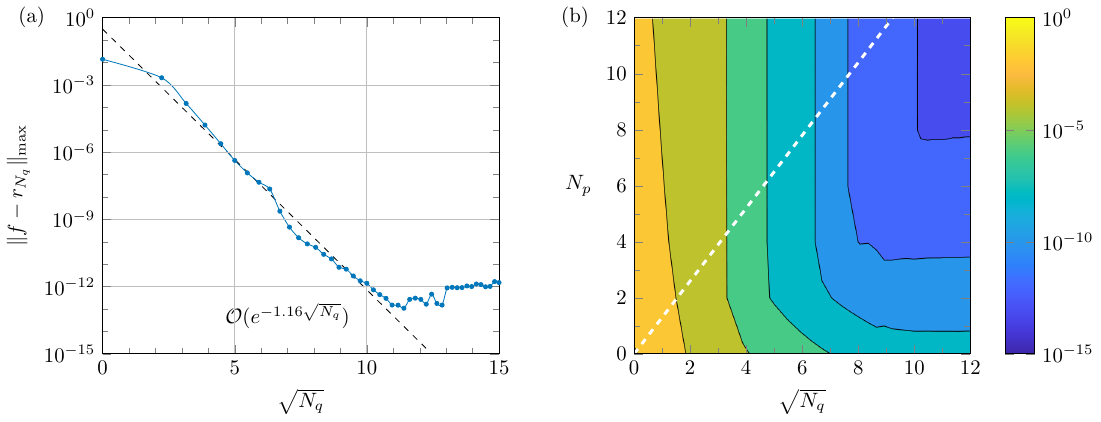}
  \end{overpic}
  \caption{(a) Convergence of the multivariate rational approximant $r$ to $f_2(x,y)=\sqrt{x+y}$ with respect to the square-root of the rational degree $N_q$. The dashed line illustrates the empirical convergence rate. (b) Approximation error between $f_2$ and its rational approximant as a function of the square root of the rational degree $N_q$ and the polynomial degree $N_p$. The white dashed line shows the scaling of the polynomial degree $N_p = 1.3\sqrt{N_q}$ employed in this work. Both figures display the maximum error on a fine independent grid, similarly distributed as the sampling grid.}
  \label{fig_convergence}
\end{figure}

Finally, we consider the multivariate rational approximation of the following function defined on the unit disk $\Omega=D(0,1)$:
\begin{equation} \label{eq_function_disk}
  f_3(r,\theta) =
  \begin{cases}
    \cos(10 r+10 \theta),            & r\in [0,3/4],\,\theta\in[-\pi,\pi), \\
    -\sqrt{1-r}\cos(10 r-10 \theta), & r\in (3/4,1],\,\theta\in[-\pi,\pi).
  \end{cases}
\end{equation}
We note that this function has a square-root singularity along the unit circle and a discontinuity at $r=3/4$. To approximate this function, we use a tensor-product rational function consisting of a mixed Chebyshev polynomial basis and rational functions in the radial direction $r$ with poles clustering towards $r=1$ and $r=3/4$ parallel to the imaginary axis, and a Fourier expansion in the angular variable $\theta$ to impose periodicity. We plot the function along with the approximation error $|f_3(x,y)-r(x,y)|$ in \cref{fig_disk}, and observe a maximum error of $3.6\times 10^{-13}$ when evaluating the rational approximant on a $1000\times 1000$ equispaced grid in the radial and angular directions.

\begin{figure}[htbp]
  \centering
  \begin{overpic}[width=\textwidth]{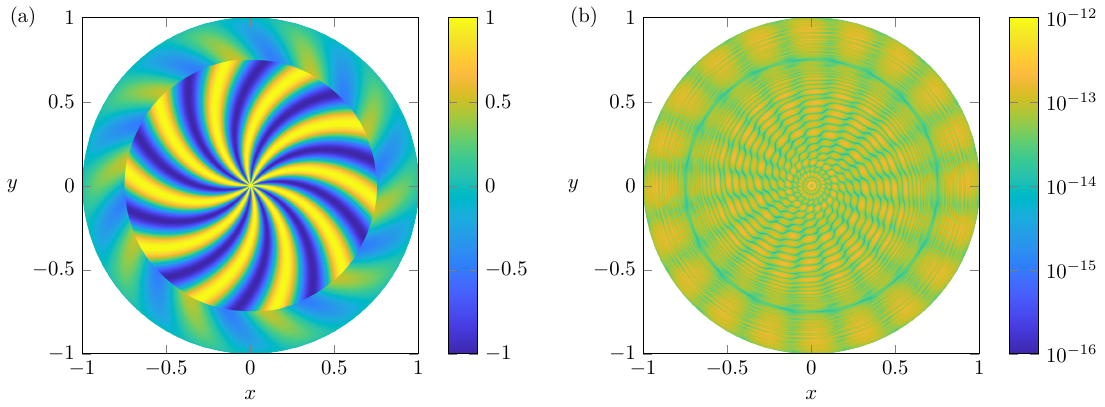}
  \end{overpic}
  \caption{(a) The function $f_3(r,\theta)$ defined in \cref{eq_function_disk} with discontinuity at $r=3/4$ and square-root singularity at $r=1$. (b) The approximation error with the rational approximant, whose poles cluster exponentially fast towards $r=3/4$ and $r=1$ from both sides in the imaginary direction.}
  \label{fig_disk}
\end{figure}

\subsection{Robust and optimal parameter choices} \label{sec:parameterchoices}
The construction proposed in~\cref{sec_construction_tensor} is designed to provide a robust approximation scheme for multivariate functions with edge singularities at known locations. In this section, we first show that the parameter choices described in~\cref{sec_construction_tensor} give satisfactory results for a variety of problems. Next, we explain how the approximation scheme can be further optimized when more information on the problem is known.

The influence of the truncation parameter $\epsilon$ used in~\cref{alg_TSVD} on the overall convergence behavior and, intrinsically related, on the size of the coefficients can be seen in~\cref{parameter_choice2} for the functions $f_1$ and $f_3$ introduced in~\cref{sec_numerical}. The maximum error is again computed on an independent grid consisting of points that are distributed similarly to the sampling points, yet which is much finer than the sampling grid. For branch point singularities, the maximum error on this grid gives a reliable indication of the uniform error on $[0,1] \times [0,1]$. For discontinuities, a small error simply indicates high precision up to machine epsilon distance away from the discontinuity. The results for $f_2$ are analogous to those of $f_1$. A comparison is made between $\epsilon = 10^{-14}$ (default, full line) and $\epsilon = 10^{-10}$ (dashed line). Most importantly,~\cref{parameter_choice2}(a) shows that the approximation achieves high precision even though the system matrices in~\cref{eq_least_square} are heavily ill-conditioned. This follows from~\cref{th:tsvd} and the fact that the coefficients of the rational approximants are small at the time of convergence, as can be seen in~\cref{parameter_choice2}(b). \cref{parameter_choice2}(b) also shows that the regularization parameter has an influence on the intermediary growth of the coefficients, a phenomenon which is analyzed in~\cite{adcock2019frames,adcock2020fna2}. In this paper a fixed value of $\epsilon = 10^{-14}$ is used, resulting in high accuracy yet suffering from a large intermediary coefficient growth. In~\cite{adcock2021frame} it is shown that this growth can be avoided by varying the regularization parameter with $N_q$.

\begin{figure}[htbp]
  \centering
  \begin{overpic}[width=\textwidth]{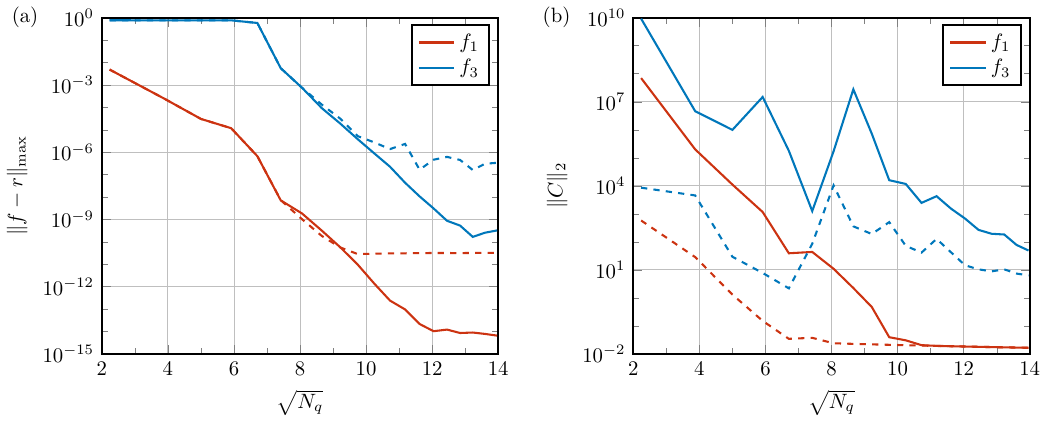}
  \end{overpic}
  \caption{(a) Approximation error between the rational approximant and the functions $f_1(x,y) = (x(1-x))^{1/4+y}\sqrt{y(1-y)}$ and $f_3(r,\theta)$ defined by~\cref{eq_function_disk} as a function of $\sqrt{N_q}$, displayed for two values of the truncated SVD parameter in \cref{alg_TSVD}: $\epsilon = 10^{-14}$ (default, full line) and $\epsilon = 10^{-10}$ (dashed line). The figure displays the maximum error on a fine independent grid, similarly distributed as the sampling grid. (b) The Frobenius norm of the coefficient matrix $C$ of the rational approximant to the functions $f_1$ and $f_3$ as a function of $\sqrt{N_q}$.  }
  \label{parameter_choice2}
\end{figure}

In~\cref{parameter_choice1}, we illustrate the influence of the parameters $\sigma$, which determines the spacing between the clustered poles, and $\epsilon$ on the accuracy of a fixed-degree rational approximant to the same functions $f_1$, $f_2$ and $f_3$. The default values $\sigma = 2\pi$ and $\epsilon = 10^{-14}$ employed in this paper are marked by a dashed line. It follows from~\cref{parameter_choice1}(a) that the approximation error shows a clear dependence on $\sigma$, materialized by a V-shaped curve for $f_1$ and $f_2$, which is similar to what has been observed in the one-dimensional case~\cite{herremans2023resolution}. This behavior stems from the balance of the truncation error for small values of $\sigma$ and the discretization error for large values of $\sigma$~\cite{herremans2023resolution,stahl1994poles,trefethen2021clustering}. The error behaves slightly differently for the discontinuous function $f_3$; the truncation error here drops abruptly once the spacing $\sigma$ is large enough such that the clustered poles cover the independent sampling grid. Of course, the pointwise error always remains ${\mathcal O}(1)$ near the discontinuity itself. \cref{parameter_choice1}(b) shows that the approximation error is nearly insensitive to the precise value of $\epsilon$, provided that it is sufficiently small compared to the achievable accuracy, which also followed from~\cref{parameter_choice1}(a).

\begin{figure}[tbp]
  \centering
  \begin{overpic}[width=\textwidth]{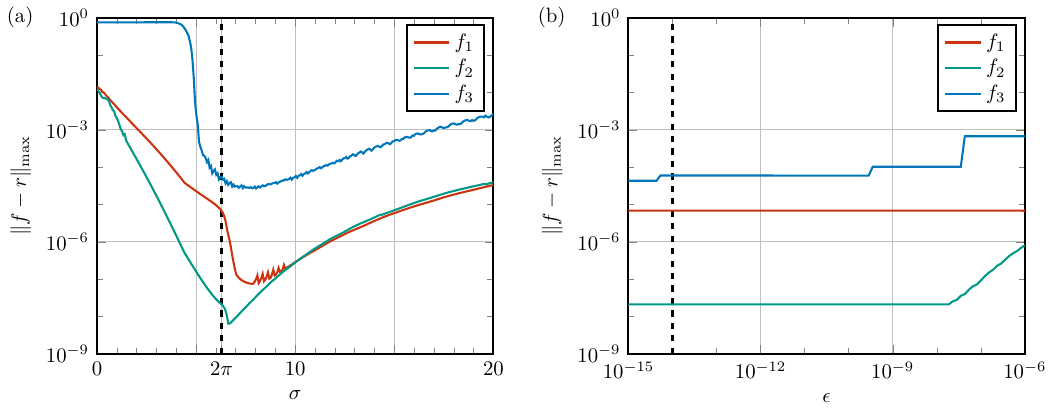}
  \end{overpic}
  \caption{(a) Evolution of the approximation error with respect to the choice of the parameter $\sigma$ in \cref{eq_poles} between the rational approximant and the functions $f_1(x,y) = (x(1-x))^{1/4+y}\sqrt{y(1-y)}$ ($N_q = 40$), $f_2(x,y)=\sqrt{x+y}$ ($N_q = 40$) and $f_3(r,\theta)$ defined by~\cref{eq_function_disk} ($N_q = 80$). The vertical dashed line highlights the choice of $\sigma=2\pi$. (b) Evolution of the approximation error with respect to the SVD truncation parameter $\epsilon$ in \cref{alg_TSVD}. The vertical dashed line highlights the choice of $\epsilon=10^{-14}$. The approximation accuracy is insensitive to the choice of $\epsilon$ in these examples once $\epsilon$ is smaller than the achievable error. Both figures display the maximum error on a fine independent grid, similarly distributed as the sampling grid.}
  \label{parameter_choice1}
\end{figure}

Prior research optimized the precise distribution of the preassigned poles in the univariate case when approximating branch point singularities $x^\alpha$. First, it follows from~\cite{trefethen2021clustering} that a tapered distribution, which we incorporated in~\cref{eq_poles}, increases the convergence rate. Furthermore, as can also be seen on~\cref{parameter_choice1}, optimizing the parameter $\sigma$ can have a significant effect on the approximation error. \cite{herremans2023resolution} shows that the optimal value for $\sigma$ is a function of the type of the singularity $\alpha$ and the angle (in the complex plane) between the approximation interval and the line on which the poles are clustered. The same dependency is observed for the clustering introduced in the multivariate approximation scheme provided that the dominant singular behavior close to the edge singularity behaves as $\mathcal{O}(d^\alpha)$, where $d$ is the distance to the edge. The most general setting is described in~\cite[Conj.~5.3]{herremans2023resolution}, where $\beta$ relates to the angle. For the multivariate approximation, this angle is defined in a plane perpendicular to the edge singularity. Note that the scheme in~\cref{sec_construction_tensor} uses poles that cluster parallel to the positive and negative imaginary axis, corresponding to $\beta = 1$. The parameter setting $\sigma = 2\pi$ used in this paper is therefore chosen such that it is optimal for square root singularities ($\alpha = 1/2$). This can be adapted if the type of the singularity is known. Moreover, if the singularity is located at the boundary of the domain, it is more efficient to use a single set of real poles outside the domain, corresponding to $\beta = 0$. Given more information on the analyticity of the function to be approximated in $\mathbb{C}^2$, a single set of imaginary poles might also suffice for singularities located in the interior of the domain. This relates to the principles discussed in~\cite[section 6]{costa2021aaa}.

The numerical scheme is not restricted to the approximation of singularities of type $\mathcal{O}(d^\alpha)$ but can be used to approximate more general \textit{local behavior}, such as discontinuities. To analyze and optimize $\sigma$ in this general setting, we believe it is beneficial to interpret and view the rational poles as sigmoidal radial basis functions on a logarithmic scale, a concept which was introduced in~\cite{huybrechs2023sigmoid}. As mentioned before, note that the proposed tapered distribution of the poles~\cref{eq_poles} is an optimization linked to the behavior of branch point singularities~\cite{huybrechs2023sigmoid,trefethen2021clustering}. For discontinuities, a uniform exponential distribution is observed to be more efficient.

\section{Piecewise rational approximation} \label{sec_piecewise}

The multivariate rational approximation algorithm described in \cref{sec_tensor} is only valid for product domains of the form $\Omega=I_1\times I_2$, where $I_1, I_2\subset\R$ are intervals, and for singularities of the function along straight lines aligning with the boundaries of $\Omega$. However, the scheme extends naturally to polygonal domains and functions with straight-line singularities using a domain decomposition technique. One motivation for the approximation of two-dimensional functions with singularities inside the domain arises from the approximation of Green's functions associated with one-dimensional differential operators. For example, the Green's functions associated with Sturm--Liouville operators on $I\subset\R$ are continuous but not differentiable along the diagonal $\mathcal{D}=\{(x,x),\, x\in I\}$ of the domain $\Omega=I\times I$~\cite[Sec.~15.2.5]{riley2006mathematical}. This requires an algorithm for approximating functions with singularities along the diagonal of the domain.

\begin{figure}[htbp]
  \centering
  \begin{overpic}[width=\textwidth]{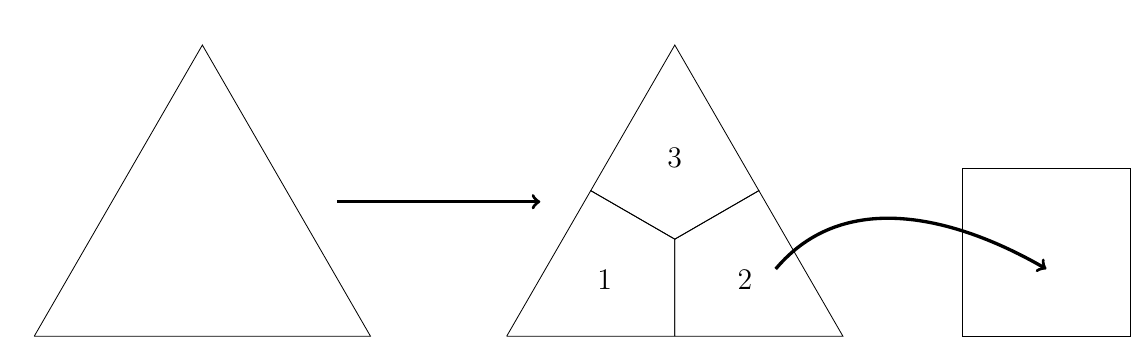}
  \end{overpic}
  \caption{Diagram illustrating the domain decomposition technique for approximating functions with singularities along straight lines. A decomposition of the domain into triangles is performed by a meshing algorithm. Then, the triangles are refined into three quadrilaterals using a barycentric refinement. Finally, each quadrilateral is mapped into a unit square reference element by an affine transformation.}
  \label{fig_diagram_decomposition}
\end{figure}

To approximate a function $f:\Omega\to \R$ with singularities located along straight lines, we decompose the domain $\Omega\subset\R^2$ into a finite number of subdomains $\Omega=\cup_{i=1}^N\Omega_i$ such that the singularities of the function are located along the boundaries of the subdomains. Here, the subdomains $\Omega_i$ are quadrilateral domains with straight-line boundaries. This can be achieved using standard meshing software such as Gmsh~\cite{geuzaine2009gmsh}, along with a barycentric refinement of the triangles into quadrilaterals (see~\cref{fig_diagram_decomposition}). Each quadrilateral composing the mesh can then be mapped into a unit square reference element by a rational transformation, for which in our implementation we chose to use the ultraSEM software system~\cite{fortunato2021ultraspherical}. Finally, we approximate the function $f$ independently on each subdomain $\Omega_i$ as
\[
  f(x,y) \approx \sum_{i=1}^N r_i(x,y)\mathds{1}_{\Omega_i}(x,y),\quad (x,y)\in\Omega,
\]
where $\mathds{1}_{\Omega_i}$ denotes the characteristic function of the domain $\Omega_i$. Here, we perform the rational approximations directly on the reference element using the tensor-product multivariate rational approximation algorithm described in \cref{sec_tensor}, and map the resulting rational functions back to the domains $\Omega_i$. The transformation to the reference element ensures that the singularities of the function $f$ are located along the boundaries of the unit square after the mapping. Hence, one can compute a rational approximation to $f$ on a domain $\Omega_i$ using the efficient tensor-product algorithm (see~\cref{sec_tensor}). As long as the number of patches is not too large, the scheme remains efficient overall.

\begin{figure}[htbp]
  \centering
  \begin{overpic}[width=\textwidth]{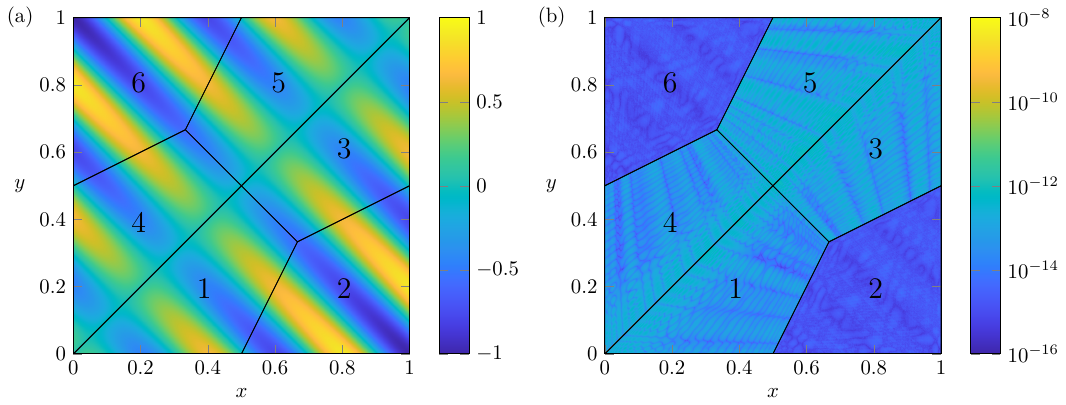}
  \end{overpic}
  \caption{(a) Piecewise rational approximation to the function $f$ defined on the unit square by \cref{eq_piecewise_f} with singularity located along the diagonal. The domain $\Omega=[0,1]^2$ is divided into six quadrilaterals so that one can efficiently perform a tensor-product rational approximation of the function on each subdomain. (b) Approximation error between the function $f$ and its piecewise rational approximant on each subdomain of the partition.}
  \label{fig_piecewise}
\end{figure}

As a numerical example, we consider the piecewise rational approximation of the function $f$ defined on $[0,1]^2$, with a singularity along the diagonal $x=y$, as
\begin{equation} \label{eq_piecewise_f}
  f(x,y) = \cos(5\pi(x+y))\sqrt{|x-y|},\quad (x,y)\in [0,1]^2.
\end{equation}
We approximate $f$ using a piecewise rational function consisting of a tensor-product rational function with poles clustering towards the diagonal $x=y$ (see \cref{fig_piecewise}). We use the method described in \cref{sec_construction_tensor} with rational degree $N_q=150$ and polynomial degree $N_p=25$ on the domains $\Omega_1, \Omega_3,\Omega_4,\Omega_5$ (see~\cref{fig_piecewise}(left)), while we perform a polynomial approximation of degree $N_p=25$ on $\Omega_2$ and $\Omega_6$ since $f$ is smooth on these domains. We plot the function along with the approximation error $|f(x,y)-r(x,y)|$ in \cref{fig_piecewise} and observe a maximum error of $1.4\times 10^{-8}$ when evaluating the rational approximant on a $500\times 500$ equispaced grid on each subdomain.

\section{Approximation of functions with curved singularities} \label{sec_curved}

Thus far, the constructions in this paper have explored rational approximations in tensor-product form. In the current section, we explore more general bivariate rational approximations. Indeed, when the function $f$ has singularities located along curved lines inside the domain $\Omega$, the piecewise rational tensor-product approximation scheme introduced in \cref{sec_piecewise} is no longer applicable as it would require a very fine refinement of the domain to approximate the singularity curve. Additionally, some applications may require a global rational approximant rather than a piecewise rational expression as in \cref{sec_piecewise}. In this section, we present a method for constructing a multivariate rational approximant to a broad class of functions with singularities along curves. For rational approximations involving such functions, we switch from the efficient tensor-product algorithm (\cref{alg_TSVD}) to a direct solver for a large discrete least squares system.

\subsection{Methodology} \label{sec_general}

Two general considerations affect the choice of a suitable approximation space. First, the location of the curve influences the denominator of the rational approximations. Here, the question arises of how to achieve clustering along curves. Second, the function to be approximated may also exhibit variation along the curve, and this leads to the question of how to represent varying residues in a numerically stable way. We show how to address both of these questions using a number of examples.

We consider functions with singularity located along curves which are precisely the zero-level curves of a polynomial $Q(x,y)$. For example, a diagonal line corresponds to $Q(x,y)=x-y$ and the unit circle to $Q(x,y)=x^2+y^2-1$. We emphasize that this definition supports multiple singularity curves in the domain through partial fraction representation of $1/Q(x,y)$, such as $Q(x,y)=(x^2+y^2-1)(x^2+y^2-(3/4)^2)$ for the function defined in \cref{eq_function_disk} and plotted in \cref{fig_disk}(a), which is singular at the circles of radius $r=1$ and $r=3/4$, as well as self-intersecting curves like $Q(x,y) = (x-1/2)(y-1/2)$. On the other hand, curves that are not the zero-level of a polynomial would still have to be approximated.

In the most general formulation, we consider a function $f:\Omega=I_1\times I_2\subset \R^2\to \R$ with polynomial singularity curve located at the roots of a polynomial $Q(x,y)$, and its approximation by a rational function of the form
\begin{equation} \label{eq_general}
  r(x,y) = \sum_{j=1}^{N_q}\sum_{0\leq k,\ell\leq N_p} a_{jkl}\frac{p_j P_k(x)P_\ell(y)}{Q(x,y)-p_j}+\sum_{0\leq k,\ell\leq N_s} b_{kl}P_k(x)P_\ell(y), \quad (x,y)\in \Omega,
\end{equation}
where the sequence $\{p_j\}_{j=1}^{N_q}$ clusters exponentially towards zero along the imaginary axis as in \cref{eq_poles}, and the polynomials $\{P_j\}_j$ are chosen to be Chebyshev polynomials (scaled to the intervals $I_1$ and $I_2$). The denominators $Q(x,y)-p_j$ may be thought of as clustering curves, in the direction normal to the singularity curve. In the four-dimensional space that arises for complex $x$ and $y$, these are actually two-dimensional manifolds. The possibility of clustering complex manifolds this way towards a curve in the real plane has been a motivating observation for this paper.

Intuitively, the corresponding residues would vary in the tangential direction, i.e., along the curves. Yet, such behavior is harder to express in general form by means of a polynomial basis for general curves. The form above resorts to a complete bivariate polynomial in tensor-product form in the numerator. One can expect this representation to be highly redundant, but we note that more efficient representations can be found in special cases. We include an example later on.

\begin{figure}[tbp]
  \centering
  \begin{overpic}[width=0.9\textwidth]{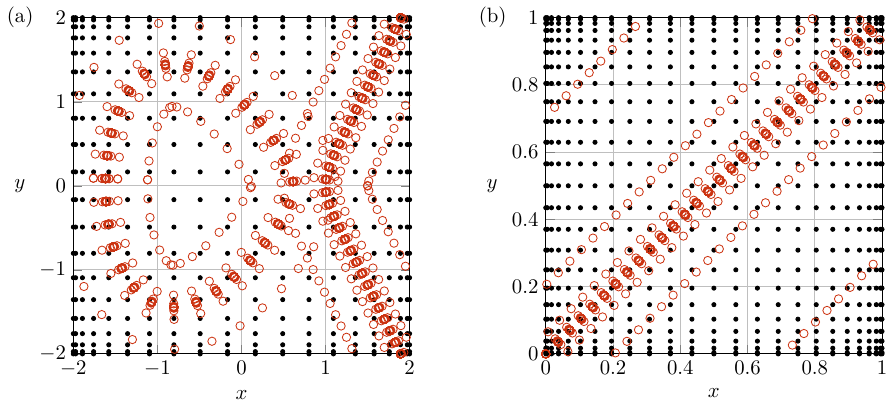}
  \end{overpic}
  \caption{Sample points used for the numerical examples in \cref{sec_elliptic,sec_green} (a)-(b), which consist of a Chebyshev grid (black dots) and points clustering exponentially fast towards the singularity curve in the normal direction (red dots).}
  \label{fig_sample_pts_curve}
\end{figure}

We select a set of sample points $\{x_i,y_i\}_{i=1}^{M}$ as the superposition of Chebyshev points and of points clustering towards the roots of $Q(x,y)$ in the domain $\Omega$, in a similar manner as \cref{sec_tensor}. Then, we solve a large least-square system of the form
\begin{equation} \label{eq_large_ls}
  \min_{\mathbf{c} \in \R^N} \|A \mathbf{c} - \mathbf{f}\|_{\textup{F}},
\end{equation}
where the matrix $A\in \R^{N\times M}$ consists of the basis of functions (rationals and polynomials) evaluated at the sample points, $\mathbf{f}=\begin{bmatrix}
    f(x_1,y_1) & \ldots & f(x_M,y_M)
  \end{bmatrix}^\top$ is the vector of sample values, and $\mathbf{c} \in \R^{N}$ is a vector containing the $N=N_q(N_p+1)^2+(N_s+1)^2$ coefficients, $\{a_{jkl}\}$ and $\{b_{kl}\}$, of $r$.

\subsection{Singularity along an elliptic curve} \label{sec_elliptic}

In this numerical example, we consider the multivariate rational approximation of the function $f$ on the domain $\Omega=[-2,2]^2$, defined as
\[
  f(x,y) = |x^3-2x+1-y^2|,\quad (x,y)\in [-2,2]^2.
\]
Following the domain of smoothness of the absolute value function, $f$ is continuous everywhere but not differentiable at the points $(x,y)\in[-2,2]^2$ where $f$ vanishes, \emph{i.e.}, along the elliptic curve defined by the algebraic equation
\[Q(x,y) = x^3-2x+1-y^2 =0,\]
whose real graph has two disconnected components since the curve is non-singular~\cite{silverman2009arithmetic}. Here, we select the rational, polynomial, and smooth residue degrees in \cref{eq_general} as $N_q=50$, $N_p=60$, $N_s=3$, respectively.

\begin{figure}[htbp]
  \centering
  \begin{overpic}[width=\textwidth]{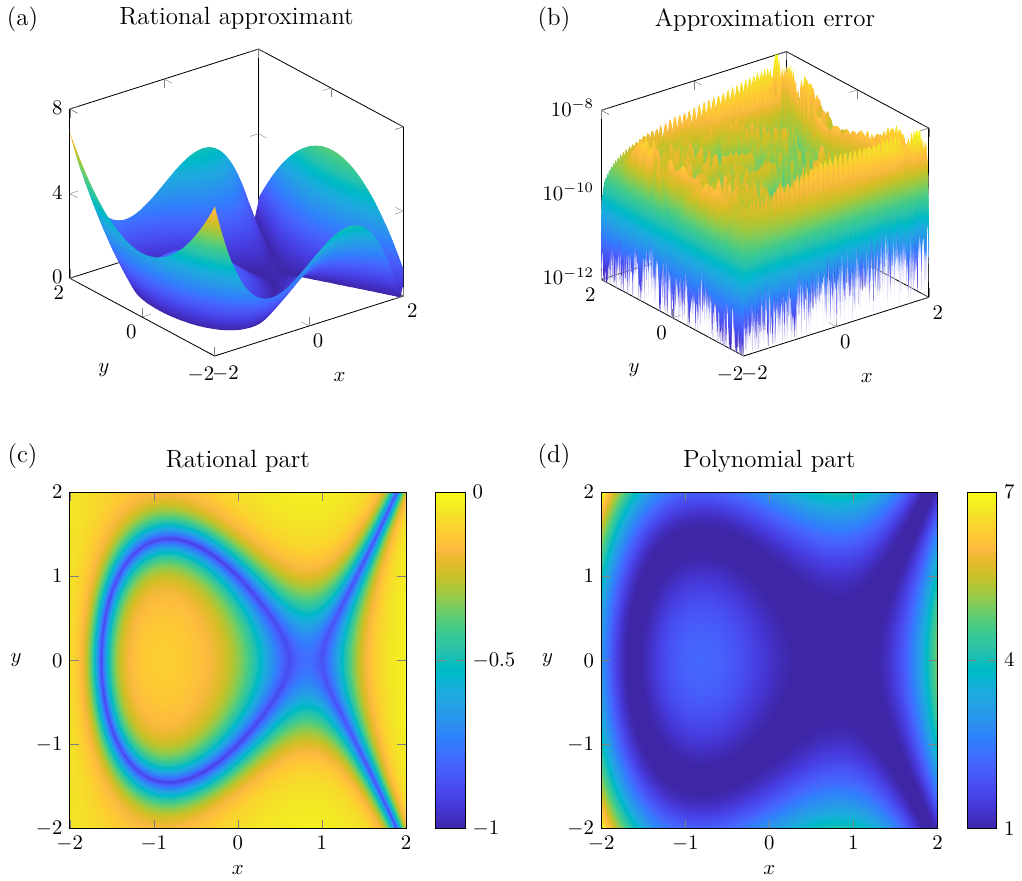}
  \end{overpic}
  \caption{(a) Multivariate rational approximant $r(x,y)$ to the function $f(x,y) = |x^3-2x+1-y^2|$, which is not differentiable along an elliptic curve, together with the approximation error $|f(x,y)-r(x,y)|$ (b). (c)-(d) Rational and smooth residue parts of the multivariate approximation.}
  \label{fig_elliptic}
\end{figure}

Following \cref{sec_general}, we evaluate the function $f$ at a grid of $M_s\times M_s$ Chebsyhev points on $[-2,2]^2$, where $M_s = 2 N_p=120$. Second, we construct a set of points that cluster towards the singularity curve defined by $Q(x,y)=0$. To do so, in our implementation, we first represent $Q$ by a chebfun2 object in the Chebfun software system~\cite{driscoll2014chebfun,townsend2013extension}. Then, we obtain a complex parameterization of the two disconnected components of the curve as $\gamma_1,\gamma_2:[-1,1]\to \mathbb{C}$ using Chebfun's rootfinding algorithm. This numerical method consists of interpolating the zero-level curves of $Q$ by using MATLAB's \texttt{contourc} command. For each disconnected component $i\in\{1,2\}$, we sample the corresponding parameterization $\gamma_i$ at $M_p=20$ uniform points $t_1,\ldots,t_{M_p}\in [-1,1]$ to obtain a set of points $\{(\Re(\gamma_i(t_j)),\Im(\gamma_i(t_j)))\}_{j=1}^{M_p}$ located along the disconnected component. For each of these points, we sample $f$ at $M_q=2 N_q$ points clustering exponentially close to the curve in each of the normal directions. We display in \cref{fig_sample_pts_curve}(a) the sample points used to evaluate the function $f$ and observe the clustering of points colored in red towards the elliptic curve.

Finally, we solve the resulting least-square system~\cref{eq_large_ls} using MATLAB's backslash command and obtain the coefficients of the rational approximant $r$ to $f$. We then evaluate the approximation error $|f(x,y)-r(x,y)|$ on a $1000\times 1000$ equispaced grid on $[-2,2]^2$, and obtain a maximum error of $1.9\times 10^{-8}$. We display in \cref{fig_elliptic} the rational approximant $r$ to $f$ (a) along with the approximation error $|f(x,y)-r(x,y)|$ (b). Moreover, we plot the rational and smooth residue parts of $r$, corresponding respectively to the two terms in~\cref{eq_general}, in panels (c) and (d).

\subsection{Diagonal singularity of a Green's function} \label{sec_green}

As noted in \cref{sec_general} a diagonal singularity curve corresponds to $Q(x,y)=x-y$. In this case, variations in the tangential direction can be expressed as a polynomial of $x+y$, since $x+y$ is constant in the normal direction. Thus, we can specialize the general form~\eqref{eq_general} to the more compact representation
\begin{equation} \label{eq_general_diag}
  r(x,y) = \sum_{j=1}^{N_q}\sum_{0\leq k \leq N_p} a_{jk}\frac{p_j P_k(x+y)}{x-y-p_j}+\sum_{0\leq k,\ell\leq N_s} b_{kl}P_k(x)P_\ell(y), \quad (x,y)\in \Omega.
\end{equation}
Note the use of $a_{jk}$ in this formula compared to $a_{jkl}$ in~\eqref{eq_general}. Unfortunately, it is not the case in general that constant functions in the direction normal to the level curve of a polynomial are themselves polynomials.

\begin{figure}[htbp]
  \centering
  \begin{overpic}[width=\textwidth]{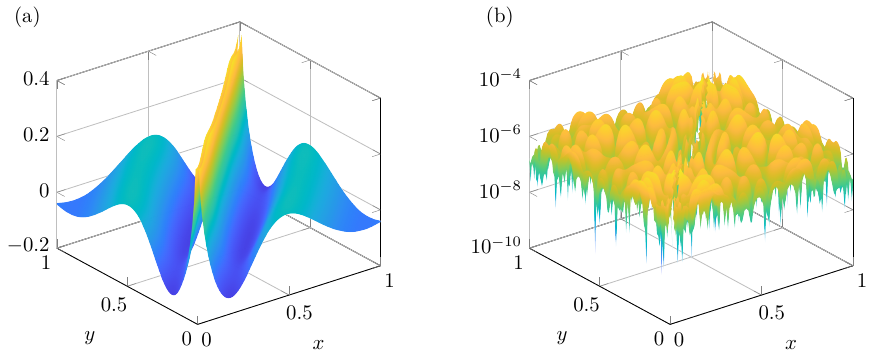}
  \end{overpic}
  \caption{Approximation of the Green's function $G(\mx,\my)$ of the gravity Helmholtz equation (as defined in~\cite{barnett2015high}). The function is singular on the diagonal and has wavelike behavior in the tangential and normal directions. The left panel shows the function for $\mx$ and $\my$ varying along a semi-circle, leading to a bivariate function on the square $[0,1]^2$. The right panel shows uniformly high accuracy of the approximation in a dense grid of points, including points close to but excluding the diagonal.}
  \label{fig_green}
\end{figure}

We illustrate the method with a non-trivial example. The Green's function of an elliptic PDE is a bivariate function $G(\mx,\my)$, which is typically singular along the diagonal $\mx=\my$. This function is known analytically for a number of common cases, such as the Laplace and Helmholtz equations, but it has to be computed numerically for differential operators with varying coefficients. In applications such as boundary element methods for integral equations, it is beneficial to have an efficient representation of the Green's function as it has to be evaluated a large number of times.

Here, we construct a multivariate rational approximation to the Green's function of the so-called \emph{gravity Helmholtz equation} which describes the propagation of waves in a 2D stratified medium~\cite{barnett2015high,slevinsky2017fast}. The function has the form~\cite[(11)]{barnett2015high}
\[
  G(\mx,\my) = A(\mx,\my) \log \frac{1}{| \mx-\my |} + B(\mx, \my).
\]
Mimicking its appearance in a boundary element method, we consider this function with $\mx, \my \in \R^2$ varying along the boundary of a domain. Thus, we consider $F(x,y) = G(\kappa(x),\kappa(y))$, in which we have chosen $\kappa$ as the parameterization of a semi-circle, with so-called energy level $E=15$ in the definition of the problem. In this experiment we choose $N_q=25$, $N_p=5$ and $N_s=15$, with sampling parameters $M_q = 2 N_q$, $M_p=2N_p$ and $M_s=2N_s$. \cref{fig_green} illustrates the approximation of the Green's function along with the approximation error on the domain. Here, we obtain a maximum error of $6\times 10^{-6}$ on a $1000\times 1000$ uniform grid.

\section{Concluding remarks}
\label{sec:conclusions}

In this paper, we explored the construction of bivariate rational approximations with fixed, well-chosen poles. When a function has singularities along straight lines, we proposed an efficient tensor-product scheme that converges to nearly machine precision accuracy at a root-exponential rate, as shown by convergence analysis and numerical experiments. The method is based on a least-squares formulation, which can be solved using effective and provably efficient regularization techniques, and generalizes naturally to piecewise rational approximations using domain decompositions. We then presented a global approximation scheme for functions with singularities along curved lines, which is demonstrated in a number of examples, including the approximation of the Green's function associated with the gravity Helmholtz equation, which is singular along the diagonal. An extension to three-dimensional functions is computationally expensive but straightforward.

Yet, there are several other directions for further research. One open problem is a more efficient solver for the large least squares problems arising in the approximation problems of \cref{sec_curved}. A second problem is the compact polynomial representation of tangentially varying residues in that setting. Finally, arguably the largest open challenge is the development of a non-linear adaptive approximation scheme to detect singularities based on function samples, without fixing poles a priori. Hence, it remains to explore methods that can achieve in multiple dimensions what AAA achieves in the univariate setting.

\section*{Code availability}
The code needed to reproduce the numerical examples in this paper is available on GitHub at \url{https://github.com/NBoulle/MultivariateRational}.

\section*{Acknowledgments}
The authors wish to thank Arne Bouillon for the feedback on the manuscript.

\appendix
\section{Separable truncated SVD algorithm} \label{sec_appendix}

An alternative and more straightforward approach to solve the least-square system~\cref{eq_least_square} compared to \cref{alg_TSVD} is to compute the truncated singular value decomposition (SVD) of the matrices $A$ and $B$ separately with a relative truncation parameter $\epsilon$ (see~\cref{alg_TSVD_sep}). This strategy is equivalent to keeping the following singular values of the Kronecker product matrix $B\otimes A$:
\[
  \{\sigma_k(A)\sigma_\ell(B)\mid \sigma_k(A)\geq \epsilon\sigma_1(A)\text{ and }\sigma_\ell(B)\geq \epsilon \sigma_1(B)\},
\]
and preserving some singular values of order $\mathcal{O}(\epsilon^2)$ due to cross-products. However, this leads to a less stable numerical solver for small values of the threshold parameter $\epsilon$, in particular, smaller than or close to the square root of machine precision.

\begin{algorithm}
  \caption{Separable truncated SVD solver}
  \label{alg_TSVD_sep}
  \begin{algorithmic}[1]
    \Require $A \in \mathbb{C}^{M_A \times N_A}$, $B \in \mathbb{C}^{M_B \times N_B}$, $F \in \mathbb{C}^{M_A \times M_B}$, relative tolerance $\epsilon > 0$
    \Ensure $C \in \mathbb{C}^{N_A \times N_B}$ such that $ACB^\top \approx F$
    \State Compute the SVD of $A$ and $B$, $A = U_A\Sigma_A V_A^*$ and $B = U_B\Sigma_B V_B^*$, where
    \begin{align*}
      U_A \Sigma_A V_A^* & = \begin{bmatrix} U_{A1} & U_{A2} \end{bmatrix} \begin{bmatrix} \Sigma_{A1} & \\ & \Sigma_{A2} \end{bmatrix} \begin{bmatrix} V_{A1} & V_{A2} \end{bmatrix}^*, \\
      U_B \Sigma_B V_B^* & = \begin{bmatrix} U_{B1} & U_{B2} \end{bmatrix} \begin{bmatrix} \Sigma_{B1} & \\ & \Sigma_{B2} \end{bmatrix} \begin{bmatrix} V_{B1} & V_{B2} \end{bmatrix}^*,
    \end{align*}
    with $0 \leq \Sigma_{A2} < \epsilon \norm{A}_2 I \leq \Sigma_{A1}$ and $0 \leq \Sigma_{B2} < \epsilon \norm{B}_2 I \leq \Sigma_{B1}$.
    \State Return $C = V_{A1} (\Sigma_{A1}^{-1}(U_{A1}^*F U_{B1}^{* \top}) \Sigma_{B1}^{-1}) V_{B1}^\top$.
  \end{algorithmic}
\end{algorithm}

Nonetheless, this approach remains theoretically rigorous. Hence, analogously to the bounds introduced in \cite{adcock2019frames,adcock2020fna2,coppe2020az}, the following theorem shows that the residual of the solution $C$ returned by~\cref{alg_TSVD_sep} is of order $\epsilon$ for functions $f$ that can be approximated to within $\epsilon$ by a tensor-product approximant, i.e., there exists a bounded coefficient matrix $X \in \mathbb{C}^{N_A \times N_B}$, satisfying $\norm{X}_{\textup{F}} \approx 1$, such that $\norm{F-AXB^\top}_{\textup{F}} \approx \epsilon$.

\begin{theorem}\label{th:tsvd_sep}
  Let $A \in \mathbb{C}^{M_A \times N_A}$, $B \in \mathbb{C}^{M_B \times N_B}$, $F \in \mathbb{C}^{M_A \times M_B}$, and choose a relative threshold parameter of $0<\epsilon < 1$. Let $C\in \mathbb{C}^{N_A\times N_B}$ be the coefficient matrix computed by Algorithm~\ref{alg_TSVD_sep}. Then,
  \[
    \norm{F - ACB^\top}_{\textup{F}} \leq \inf_{X \in \mathbb{C}^{N_A \times N_B}} \left\{ 2 \norm{F-AXB^\top}_{\textup{F}} + \gamma_{A,B}^\epsilon \; \epsilon \norm{X}_{\textup{F}} \right\},
  \]
  where $\gamma_{A,B}^\epsilon \coloneqq (2 + \epsilon) \norm{A}_2 \norm{B}_2$.
\end{theorem}
\begin{proof}
  We substitute $$C = V_{A1} (\Sigma_{A1}^{-1}(U_{A1}^*F U_{B1}^{* \top}) \Sigma_{B1}^{-1}) V_{B1}^\top = (V_{A1}\Sigma_{A1}^{-1}U_{A1}^*)F (V_{B1}\Sigma_{B1}^{-1}U_{B1}^*)^\top$$ into the residual of \cref{eq_least_square} to obtain
  \[
    \norm{F-ACB^\top}_{\textup{F}} = F - A(V_{A1}\Sigma_{A1}^{-1}U_{A1}^*)F (V_{B1}\Sigma_{B1}^{-1}U_{B1}^*)^\top B^\top.
  \]
  One can expand the block form of the SVD of $A$ into $A = U_{A1}\Sigma_{A1}V_{A1}^* + U_{A2}\Sigma_{A2}V_{A2}^*$ and $B$ into $B = U_{B1}\Sigma_{B1}V_{B1}^* + U_{B2}\Sigma_{B2}V_{B2}^*$, with $V_{A1}^*V_{A2} = 0$ and $V_{B1}^*V_{B2} = 0$. Therefore,
  \begin{align*}
    AV_{A1}\Sigma_{A1}^{-1}U_{A1}^* & = U_{A1}\Sigma_{A1}V_{A1}^*V_{A1}\Sigma_{A1}^{-1}U_{A1}^* + U_{A2}\Sigma_{A2}V_{A2}^*V_{A1}\Sigma_{A1}^{-1}U_{A1}^* \\
                                    & = U_{A1}\Sigma_{A1}\Sigma_{A1}^{-1}U_{A1}^* = U_{A1}U_{A1}^*
  \end{align*}
  and similarly
  \begin{align*}
    (V_{B1}\Sigma_{B1}^{-1}U_{B1}^*)^\top B^\top = (B V_{B1}\Sigma_{B1}^{-1}U_{B1}^*)^\top = (U_{B1}U_{B1}^*)^\top.
  \end{align*}
  For any $X \in \mathbb{C}^{N_A \times N_B}$, we can add and subtract $AXB^\top - U_{A1}U_{A1}^* (AXB^\top) (U_{B1} U_{B1}^*)^\top$ to get
  \begin{align*}
    F-ACB^\top = & \underbrace{(F - AXB^\top) - U_{A1}U_{A1}^* (F - AXB^\top) (U_{B1} U_{B1}^*)^\top}_{1} \\ & +\underbrace{AXB^\top - U_{A1}U_{A1}^* (AXB^\top) (U_{B1} U_{B1}^*)^\top}_{2}.
  \end{align*}
  Using the standard inequalities $\norm{AB}_{\textup{F}} \leq \norm{A}_2 \norm{B}_{\textup{F}}$ and  $\norm{AB}_{\textup{F}} \leq \norm{A}_{\textup{F}} \norm{B}_2$~\cite[\S50.3]{hogben2006handbook}, the Frobenius norm of the first term can be bounded by
  \begin{align*}
     & \norm{(F - AXB^\top) - U_{A1}U_{A1}^* (F - AXB^\top) (U_{B1} U_{B1}^*)^\top}_{\textup{F}}
    \\ &\quad\leq (1 + \norm{U_{A1} U_{A1}^*}_2 \norm{(U_{B1} U_{B1}^*)^\top}_2) \norm{F - AXB^\top}_{\textup{F}} \\ &\quad\leq 2 \norm{F - AXB^\top}_{\textup{F}}.
  \end{align*}
  Again using the expansion of the block form of the SVD, the second term becomes
  \begin{align*}
     & AXB^\top - U_{A1}U_{A1}^* (AXB^\top) (U_{B1} U_{B1}^*)^\top                                                                  \\
     & \quad = (U_{A1}\Sigma_{A1}V_{A1}^* + U_{A2}\Sigma_{A2}V_{A2}^*)X(U_{B1}\Sigma_{B1}V_{B1}^* + U_{B2}\Sigma_{B2}V_{B2}^*)^\top \\ &\qquad - (U_{A1}\Sigma_{A1}V_{A1}^*)X(U_{B1}\Sigma_{B1}V_{B1}^*)^\top
    \\ &\quad = (U_{A1}\Sigma_{A1}V_{A1}^*)X(U_{B2}\Sigma_{B2}V_{B2}^*)^\top + (U_{A2}\Sigma_{A2}V_{A2}^*)X(U_{B1}\Sigma_{B1}V_{B1}^*)^\top \\ &\qquad + (U_{A2}\Sigma_{A2}V_{A2}^*)X(U_{B2}\Sigma_{B2}V_{B2}^*)^\top.
  \end{align*}
  Therefore,
  \begin{align*}
     & \norm{AXB^\top - U_{A1}U_{A1}^* (AXB^\top) (U_{B1} U_{B1}^*)^\top}_{\textup{F}} \\ &\quad\leq (\norm{\Sigma_{A1}}_2\norm{\Sigma_{B2}}_2 + \norm{\Sigma_{A2}}_2\norm{\Sigma_{B1}}_2 + \norm{\Sigma_{A2}}_2\norm{\Sigma_{B2}}_2) \norm{X}_{\textup{F}}.
  \end{align*}
  The final bound follows from
  \[
    \norm{\Sigma_{A1}}_2 \leq \norm{A}_2, \quad \norm{\Sigma_{A2}}_2 \leq \epsilon \norm{A}_2, \quad \norm{\Sigma_{B1}}_{\textup{F}} \leq \norm{B}_2, \quad \norm{\Sigma_{B2}}_{\textup{F}} \leq \epsilon \norm{B}_2.
  \]
\end{proof}

\begin{remark}
  Algorithm~\ref{alg_TSVD_sep} and Theorem~\ref{th:tsvd_sep} can also be formulated for 3D or higher-dimensional approximations with only minor changes. The rationale is that the regularization is applied in each dimension independently.
\end{remark}

\section{Convergence of the univariate lightning approximation} \label{sec_appendix2}

The following two lemmas correspond to Theorem 2.2 and Theorem 2.3 of \cite{gopal2019solving}. The first lemma is a local statement about approximation near a singularity, while the second one describes approximation on $[0,1]$. The results were not explicitly stated in \cite{gopal2019solving} as we formulate them here, hence we include short proofs. The first lemma specifies an a priori choice of poles in the unit disk $D(0,1)$ on the interval $[-1,0]$.

\begin{lemma} \label{lem:gopal2019solving_lemma1}
  Let $f$ satisfy the conditions of \cite[Thm.~2.2]{gopal2019solving}, and define all symbols as in that theorem, in particular the poles $p_j = e^{-\sigma j/\sqrt{n}}$ for some $\sigma > 0$ and $0 \leq j \leq n-1$. Assume in addition that $|f(z)| \leq D |z|^\delta$ with $D>0$ and $z$ in a disk with radius $b$ around the origin. There exists an $n_0 > 0$ sufficiently large, depending on $b>0$ but not on $f$, and a type $(n-1,n)$ rational function $r_n$ such that for $z \in [0,\rho]$
  \[
    |f(z) - r_n(z)| \leq A_0 \, D \, \Vert f \Vert_{\infty,D(0,1)} \, e^{-C \sqrt{n}}, \quad n > n_0,
  \]
  where $A_0$ is independent of $f$. Here, the constant $C$ depends only on $\delta$ and on $\sigma$, while $\rho$ is sufficiently small and satisfies $0<\rho<1$ and depends only on $\sigma$.
\end{lemma}

\begin{proof}
  The proof of \cite[Thm.~2.2]{gopal2019solving} shows that constants $A$ and $C$ exist such that
  \[
    | f(z) - r(z)| \leq A e^{-C \sqrt{n}}, \quad n \geq 1,
  \]
  for $z$ in a wedge centered at $z=0$ with angle $\theta$ and radius $0 < \rho < 1$. For the constant $C$, we merely note that it is determined by the distribution of poles and by the value of $\delta$. The value of $\rho$ depends only on the distribution of poles, i.e., on $\sigma$. We can arbitrarily choose the angle $\theta$ of the wedge, as we are only interested in its intersection $[0,\rho]$ with the real line.

  The constant $A$ is more involved. In the cited proof, it arises from two applications of a H{\"o}lder inequality, giving rise to $A = A_1 + A_2$. The first application uses $|f(z)| = {\mathcal O}(|z|^\delta)$ on the interval $[-\epsilon,0]$, in which $-\epsilon = -e^{-\sigma(n-1)/\sqrt{n}}$ is the smallest pole. We take $n_0$ sufficiently large so that we may use the bound $|f(z)| \leq D |z|^\delta$ instead, showing that $A_1$ is proportional to $D$ but independent of $f$.

  The second application leads to $A_2$ being proportional to
  $\max_{z \in [0,\rho]} \int_{\Gamma_1} |\frac{f(t)}{t-z} |\d t$,
  in which $\Gamma_1$ is the unit circle with $-1$ connected to $-\epsilon$ both above and below the slit~\cite[Fig.~2]{gopal2019solving}. The integrand can be bounded by substituting $f$ with $ \Vert f \Vert_{\infty,D(0,1)}$, if $\lvert t \rvert > b$, and by $D \rvert t \lvert^\delta$, if $\lvert t \rvert \leq b$. The resulting bound is proportional to these constants. Note that the different bound in the disk of radius $b$ is necessary for the Cauchy integral to be bounded uniformly in $n$ since the integral is finite as $-\epsilon$ approaches zero from the left and $z$ approaches zero from the right only because $f$ vanishes at the origin.
\end{proof}

The second lemma involves a scaling with a factor $R > 1$, which in principle maps those poles to a larger interval $[-R,0]$. However, both in theory and in practice, the poles can be scaled essentially to an arbitrary interval near the singularity. On the one hand, what matters is not the precise formula for the poles but their density close to the singularity. On the other hand, rational functions with poles further away can always be absorbed into the polynomial part, hence the scaling is not critical. The rate of convergence $C$ is more sensitive to the choice of $\sigma$, which governs the density of small poles, than it is to other parameters, but convergence is root-exponential no matter how $\sigma$ and the scaling of the poles are chosen.

\begin{lemma} \label{lem:gopal2019solving_lemma2}
  Let $f$ be analytic on a complex neighborhood $\Omega$ of $[0,1]$ but excluding a slit along the negative real axis up to and including the point $z=0$. Assume that $f$ satisfies $|f(z)| \leq D |z|^\delta$ for some $\delta > 0$ and $z$ in a disk of radius $b>0$ around the origin. Then, there exist type $(n+m,n)$ rational approximations $r_{n+m,n}$ and a constant $C > 0$, independent of $f$, such that
  \[
    \Vert f - r_{n+m,n} \Vert_{\infty,[0,1]} \leq A \, D\, \Vert f \Vert_{\infty,\Omega} \, e^{-C \sqrt{n}}, \quad n > n_0,
  \]
  for some sufficiently large $n_0 > 0$. Moreover, the polynomial part can be chosen to have degree $m = E \sqrt{n}$, for a sufficiently large constant $E>0$ that also scales with $\Vert f \Vert_{\infty,\Omega}$, and the rational functions can be chosen to have simple poles at
  \begin{equation} \label{eq:scaledpoles}
    \beta_j = -R\; e^{-\sigma j / \sqrt{n}}, \quad 0 \leq j \leq n-1,
  \end{equation}
  with $\sigma > 0$ and for some sufficiently large $R > 1$ independent of $f$. The constant $A$ depends on the region of analyticity $\Omega$ and on the radius $b$ but is independent of $f$ and of $n$.
\end{lemma}
\begin{proof}
  Root-exponential convergence of rational approximations of the required form is proven in \cite[Thm.~2.3]{gopal2019solving} for $f$ analytic on a polygonal region $\Omega\subset\mathbb{C}$, with slits along the exterior bisector at corners. The current statement differs from the stated result in two ways: (i) we take $\Omega$ to be $[0,1]$ and we consider possible lack of analyticity only at the left endpoint $z=0$, and (ii) we make the dependence on $f$ more explicit. The method of proof is entirely the same as that of \cite[Theorem 2.3]{gopal2019solving} and here we focus mainly on these two differences.

  \begin{figure}
    \centering
    \includegraphics[width=.5\linewidth]{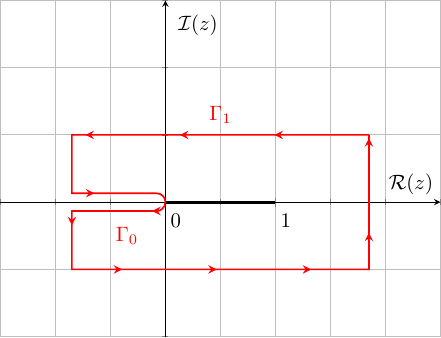}
    \caption{The approximation problem of~\cref{lem:gopal2019solving_lemma2}. The contour $\Gamma = \Gamma_0\cup\Gamma_1$ lies in the closure of the region of analyticity of $f$. }
    \label{fig:newproof}
  \end{figure}

  For the region, we substitute the contour $\Gamma$ and its constituent parts of \cite[Thm.~2.3]{gopal2019solving} by those shown in~\cref{fig:newproof}. Thus, $\Gamma_0$ is a contour from $z=-\eta$ to $z=0$ and back along both sides of the slit and the contour $\Gamma_1$ encircles $[0,1]$ with positive orientation, while both $\eta$ and the distance of $\Gamma_1$ to the real line are small enough for $\Gamma$ to remain within the region of analyticity $\Omega$ of $f$. We can write $f = f_0 + f_1$ with
  \[
    f_0 = \frac{1}{2\pi i} \int_{\Gamma_0} \frac{f(t)}{t-z}\d t \quad \mbox{and} \quad f_1 = \frac{1}{2\pi i} \int_{\Gamma_1} \frac{f(t)}{t-z}\d t.
  \]

  As in the proof of \cite[Thm.~2.3]{gopal2019solving}, the function $f_1$ is analytic away from $\Gamma_1$ in some open region containing $[0,1]$. Hence it can be approximated with exponential convergence by polynomials, with a bound satisfying~\cite[Thm.~8.2]{trefethen2019atap}
  \[
    \Vert f_1 - p_m \Vert_{\infty,[0,1]} \leq \frac{2 M \xi^{-m}}{\xi-1},
  \]
  with $\xi > 0$ independent of $f$ and with $M$ the maximum of $f_1$ along the Bernstein ellipse with radius $\xi$ (within the region of analyticity of $f_1$). This results in root-exponential convergence in $n$ of this ``smooth part'' of the problem when the polynomial degree is chosen to be $m = E \sqrt{n}$ for some $E > 0$. The factor in front depends on $M$, which can readily be bounded in terms of the maximum of $f$ in $\Omega$.

  The function $f_0$ is analytic away from $\Gamma_0$ but may be singular at the endpoint $z=-\eta$ of the contour. In the proof of \cite[Thm.~2.3]{gopal2019solving} the local result \cite[Thm.~2.2]{gopal2019solving} is invoked with a scaling factor $1/\rho$ such that the result of that theorem on $[0,\rho]$ implies the approximation of $f_0$ on $[0,1]$. The analog in the current proof is to invoke the local result of \cref{lem:gopal2019solving_lemma1} instead. However, the singularity at $z = -\eta$ is a problem for the factor $\Vert f \Vert_{\infty, D(0,1)}$ in the error bound. Hence, we scale by $R > 1/\min(\eta,\rho)$ instead, ensuring that \cref{lem:gopal2019solving_lemma1} implies approximation accuracy on an interval at least as large as $[0,1]$, while $\tilde{f}_0(z) = f_0(R z)$ is bounded on the unit disk. Furthermore, as in the proof of \cite[Thm.~2.3]{gopal2019solving}, $\tilde{f}_0(z) = \mathcal{O}(\lvert z \rvert^\delta)$ as $z \to 0$ can always be ensured by subtracting a constant from $\tilde{f_0}$, to be absorbed in the polynomial term. Though the bound $|\tilde{f}_0| \leq D |z|^\delta$ in \cref{lem:gopal2019solving_lemma1} may involve different constants from those of $f$ with the same name in the current theorem, and in particular the radius in which the bound hold may be a factor $R$ smaller, the precise relationship between the constants is fixed and independent of $f$, and the bound itself scales linearly with $D$ of the current theorem. Finally, we note that $\|\tilde{f}_0\|_{\infty,D(0,1)}$ can be bounded in terms of $\|f\|_{\infty,\Omega}$.
\end{proof}

\bibliographystyle{siamplain}
\bibliography{references}

\begin{thebibliography}{10}

\bibitem{adcock2019frames}
{\sc B.~Adcock and D.~Huybrechs}, {\em Frames and numerical approximation},
  SIAM Rev., 61 (2019), pp.~443--473.

\bibitem{adcock2020fna2}
{\sc B.~Adcock and D.~Huybrechs}, {\em Frames and numerical approximation {II}:
  {G}eneralized sampling}, J. Fourier Anal. Appl., 26 (2020), pp.~1--34.

\bibitem{adcock2021frame}
{\sc B.~Adcock and M.~Seifi}, {\em Frame approximation with bounded
  coefficients}, Adv. Comput. Math., 47 (2021), pp.~1--25.

\bibitem{austin2020practical}
{\sc A.~P. Austin, M.~Krishnamoorthy, S.~Leyffer, S.~Mrenna, J.~Müller, and
  H.~Schulz}, {\em Practical algorithms for multivariate rational
  approximation}, Comput. Phys. Commun., 261 (2021), p.~107663.

\bibitem{baddoo2020lightning}
{\sc P.~J. Baddoo}, {\em Lightning solvers for potential flows}, Fluids, 5
  (2020), p.~227.

\bibitem{barnett2015high}
{\sc A.~H. Barnett, B.~J. Nelson, and J.~M. Mahoney}, {\em High-order boundary
  integral equation solution of high frequency wave scattering from obstacles
  in an unbounded linearly stratified medium}, J. Comput. Phys., 297 (2015),
  pp.~407--426.

\bibitem{berrut2021linear}
{\sc J.-P. Berrut and G.~Elefante}, {\em A linear barycentric rational
  interpolant on starlike domains}, arXiv preprint arXiv:2104.09246,  (2021).

\bibitem{boulle2022data}
{\sc N.~Boull{\'e}, C.~J. Earls, and A.~Townsend}, {\em {Data-driven discovery
  of Green's functions with human-understandable deep learning}}, Sci. Rep., 12
  (2022), p.~4824.

\bibitem{boulle2020rational}
{\sc N.~Boull{\'e}, Y.~Nakatsukasa, and A.~Townsend}, {\em Rational neural
  networks}, in Adv. Neural Inf. Process. Syst., vol.~33, 2020,
  pp.~14243--14253.

\bibitem{brubeck2022lightning}
{\sc P.~D. Brubeck and L.~N. Trefethen}, {\em {Lightning Stokes solver}}, SIAM
  J. Sci. Comput., 44 (2022), pp.~A1205--A1226.

\bibitem{candes2002curvelets}
{\sc E.~J. Cand{\`e}s and D.~L. Donoho}, {\em New tight frames of curvelets and
  optimal representations of objects with $c^2$ singularities}, Comm. Pure
  Appl. Math., 57 (2002), pp.~219--266.

\bibitem{coppe2020az}
{\sc V.~Copp{\'e}, D.~Huybrechs, R.~Matthysen, and M.~Webb}, {\em The {AZ}
  algorithm for least squares systems with a known incomplete generalized
  inverse}, SIAM J. Matrix Anal. Appl., 41 (2020), pp.~1237--1259.

\bibitem{costa2021aaa}
{\sc S.~Costa and L.~N. Trefethen}, {\em {AAA-least squares rational
  approximation and solution of Laplace problems}}, in Proc. EMS, 2023.

\bibitem{cuyt2010practical}
{\sc A.~Cuyt and X.~Yang}, {\em A practical error formula for multivariate
  rational interpolation and approximation}, Numer. Alg., 55 (2010),
  pp.~233--243.

\bibitem{cuyt1983multivariate}
{\sc A.~A.~M. Cuyt}, {\em {Multivariate Pad{\'e}-approximants}}, J. Math. Anal.
  Appl., 96 (1983), pp.~283--293.

\bibitem{cuyt1985multivariate}
{\sc A.~A.~M. Cuyt and B.~M. Verdonk}, {\em Multivariate rational
  interpolation}, Computing, 34 (1985), pp.~41--61.

\bibitem{devore11983maximal}
{\sc R.~A. DeVore}, {\em Maximal functions and their application to rational
  approximation}, in CBMS Regional Conf. Ser. Math., vol.~3, Amer. Math. Soc.,
  1983, pp.~143--155.

\bibitem{devore1986multivariate}
{\sc R.~A. DeVore and X.-M. Yu}, {\em Multivariate rational approximation},
  Trans. Amer. Math. Soc., 293 (1986), pp.~161--169.

\bibitem{driscoll2014chebfun}
{\sc T.~A. Driscoll, N.~Hale, and L.~N. Trefethen}, {\em Chebfun guide}, 2014.

\bibitem{driscoll2024aaa}
{\sc T.~A. Driscoll, Y.~Nakatsukasa, and L.~N. Trefethen}, {\em {AAA} rational
  approximation on a continuum}, SIAM J. Sci. Comput., 46 (2024),
  pp.~A929--A952.

\bibitem{engl1996regularization}
{\sc H.~W. Engl, M.~Hanke, and A.~Neubauer}, {\em Regularization of inverse
  problems}, vol.~375, Springer Science \& Business Media, 1996.

\bibitem{evans2010partial}
{\sc L.~C. Evans}, {\em Partial differential equations}, American Mathematical
  Society, 2nd~ed., 2010.

\bibitem{fortunato2021ultraspherical}
{\sc D.~Fortunato, N.~Hale, and A.~Townsend}, {\em The ultraspherical spectral
  element method}, J. Comput. Phys., 436 (2021), p.~110087.

\bibitem{gawlik2019zolotarev}
{\sc E.~S. Gawlik}, {\em Zolotarev iterations for the matrix square root}, SIAM
  J. Matrix Anal. Appl., 40 (2019), pp.~696--719.

\bibitem{gawlik2020zolotarev}
{\sc E.~S. Gawlik and Y.~Nakatsukasa}, {\em Zolotarev's fifth and sixth
  problems}, arXiv preprint arXiv:2011.10877,  (2020).

\bibitem{gawlik2021approximating}
{\sc E.~S. Gawlik and Y.~Nakatsukasa}, {\em Approximating the pth root by
  composite rational functions}, J. Approx. Theory, 266 (2021), p.~105577.

\bibitem{geuzaine2009gmsh}
{\sc C.~Geuzaine and J.-F. Remacle}, {\em {Gmsh: A 3-D finite element mesh
  generator with built-in pre-and post-processing facilities}}, Int. J. Numer.
  Methods Eng., 79 (2009), pp.~1309--1331.

\bibitem{gopal2019new}
{\sc A.~Gopal and L.~N. Trefethen}, {\em {New Laplace and Helmholtz solvers}},
  Proc. Natl. Acad. Sci. U.S.A., 116 (2019), pp.~10223--10225.

\bibitem{gopal2019solving}
{\sc A.~Gopal and L.~N. Trefethen}, {\em {Solving Laplace problems with corner
  singularities via rational functions}}, SIAM J. Numer. Anal., 57 (2019),
  pp.~2074--2094.

\bibitem{halko2011finding}
{\sc N.~Halko, P.-G. Martinsson, and J.~A. Tropp}, {\em Finding structure with
  randomness: Probabilistic algorithms for constructing approximate matrix
  decompositions}, SIAM Rev., 53 (2011), pp.~217--288.

\bibitem{hansen2013least}
{\sc P.~C. Hansen, V.~Pereyra, and G.~Scherer}, {\em Least squares data fitting
  with applications}, JHU Press, 2013.

\bibitem{herremans2023resolution}
{\sc A.~Herremans, D.~Huybrechs, and L.~N. Trefethen}, {\em {Resolution of
  Singularities by Rational Functions}}, SIAM J. Numer. Anal., 61 (2023),
  pp.~2580--2600.

\bibitem{hogben2006handbook}
{\sc L.~Hogben}, {\em Handbook of linear algebra}, CRC press, 2006.

\bibitem{huybrechs2023aaa}
{\sc D.~Huybrechs and L.~N. Trefethen}, {\em {AAA interpolation of equispaced
  data}}, BIT Numer. Math., 63 (2023), p.~21.

\bibitem{huybrechs2023sigmoid}
{\sc D.~Huybrechs and L.~N. Trefethen}, {\em Sigmoid functions and multiscale
  resolution of singularities}, arXiv preprint arXiv:2303.01967,  (2023).

\bibitem{karniadakis2021physics}
{\sc G.~E. Karniadakis, I.~G. Kevrekidis, L.~Lu, P.~Perdikaris, S.~Wang, and
  L.~Yang}, {\em Physics-informed machine learning}, Nat. Rev. Phys., 3 (2021),
  pp.~422--440.

\bibitem{mao2020physics}
{\sc Z.~Mao, A.~D. Jagtap, and G.~E. Karniadakis}, {\em Physics-informed neural
  networks for high-speed flows}, Comput. Methods Appl. Mech. Eng., 360 (2020),
  p.~112789.

\bibitem{marcati2023exponential}
{\sc C.~Marcati, J.~A. Opschoor, P.~C. Petersen, and C.~Schwab}, {\em
  {Exponential ReLU neural network approximation rates for point and edge
  singularities}}, Found. Comput. Math., 23 (2023), pp.~1043--1127.

\bibitem{martinsson2020randomized}
{\sc P.-G. Martinsson and J.~A. Tropp}, {\em Randomized numerical linear
  algebra: Foundations and algorithms}, Acta Numer., 29 (2020), pp.~403--572.

\bibitem{nakatsukasa2018aaa}
{\sc Y.~Nakatsukasa, O.~S{\`e}te, and L.~N. Trefethen}, {\em {The AAA algorithm
  for rational approximation}}, SIAM J. Sci. Comput., 40 (2018),
  pp.~A1494--A1522.

\bibitem{nakatsukasa2020algorithm}
{\sc Y.~Nakatsukasa and L.~N. Trefethen}, {\em An algorithm for real and
  complex rational minimax approximation}, SIAM J. Sci. Comput., 42 (2020),
  pp.~A3157--A3179.

\bibitem{neumaier1998solving}
{\sc A.~Neumaier}, {\em {Solving ill-conditioned and singular linear systems: A
  tutorial on regularization}}, SIAM Rev., 40 (1998), pp.~636--666.

\bibitem{newman1964}
{\sc D.~J. Newman}, {\em Rational approximation to $|x|$}, Mich. Math. J., 11
  (1964), pp.~11--14.

\bibitem{riley2006mathematical}
{\sc K.~F. Riley, M.~P. Hobson, and S.~J. Bence}, {\em Mathematical methods for
  physics and engineering}, Cambridge University Press, 3rd~ed., 2006.

\bibitem{silverman2009arithmetic}
{\sc J.~H. Silverman}, {\em The arithmetic of elliptic curves}, Springer, 2009.

\bibitem{slevinsky2017fast}
{\sc R.~M. Slevinsky and S.~Olver}, {\em A fast and well-conditioned spectral
  method for singular integral equations}, J. Comput. Phys., 332 (2017),
  pp.~290--315.

\bibitem{stahl1993best}
{\sc H.~Stahl}, {\em Best uniform rational approximation of $x^\alpha$ on $[0,
  1]$}, Bull. Am. Math. Soc., 28 (1993), pp.~116--122.

\bibitem{stahl1994poles}
{\sc H.~Stahl}, {\em Poles and zeros of best rational approximants of $|x|$},
  Constr. Approx., 10 (1994), pp.~469--522.

\bibitem{stephan1988singularities}
{\sc E.~Stephan and J.~Whiteman}, {\em {Singularities of the Laplacian at
  corners and edges of three-dimensional domains and their treatment with
  finite element methods}}, Math. Methods Appl. Sci., 10 (1988), pp.~339--350.

\bibitem{townsend2013extension}
{\sc A.~Townsend and L.~N. Trefethen}, {\em {An extension of Chebfun to two
  dimensions}}, SIAM J. Sci. Comput., 35 (2013), pp.~C495--C518.

\bibitem{trefethen2017multivariate}
{\sc L.~Trefethen}, {\em Multivariate polynomial approximation in the
  hypercube}, Proc. Amer. Math. Soc., 145 (2017), pp.~4837--4844.

\bibitem{trefethen2019atap}
{\sc L.~N. Trefethen}, {\em {Approximation Theory and Approximation Practice}},
  SIAM, Philadelphia, extended~ed., 2019.

\bibitem{trefethen2021clustering}
{\sc L.~N. Trefethen, Y.~Nakatsukasa, and J.~A.~C. Weideman}, {\em Exponential
  node clustering at singularities for rational approximation, quadrature, and
  {{PDEs}}}, Numer. Math., 147 (2021), pp.~227--254.

\bibitem{van1993approximation}
{\sc C.~F. Van~Loan and N.~Pitsianis}, {\em {Approximation with Kronecker
  products}}, in Linear Algebra for Large Scale and Real-Time Applications,
  Springer, 1993, pp.~293--314.

\bibitem{xue2023computation}
{\sc Y.~Xue, S.~L. Waters, and L.~N. Trefethen}, {\em {Computation of
  two-dimensional Stokes flows via lightning and AAA rational approximation}},
  SIAM J. Sci. Comput., 46 (2024), pp.~A1214--A1234.

\bibitem{yang2023fast}
{\sc S.~Yang and S.~Xiang}, {\em Fast barycentric rational interpolations for
  complex functions with some singularities}, Calcolo, 60 (2023), p.~55.

\bibitem{zolotarev1877application}
{\sc E.~Zolotarev}, {\em Application of elliptic functions to questions of
  functions deviating least and most from zero}, Zap. Imp. Akad. Nauk. St.
  Petersburg, 30 (1877), pp.~1--59.

\end{thebibliography}
\end{document}